\documentclass[11pt,a4paper,reqno]{amsart}

\usepackage[T1]{fontenc}

\usepackage[utf8]{inputenc}

\usepackage{amsmath,amsthm,amsopn,amsfonts,amssymb,color}
\usepackage{cancel}
 \usepackage{tikz}
\usetikzlibrary{intersections,calc}
\usepackage{graphicx}
\usepackage{float}
\usepackage{amsmath}
\usepackage{fancyhdr}
\usepackage{epstopdf}
\usepackage{amsfonts}
\usepackage{pgfplots}
\pgfplotsset{compat=newest}
\usetikzlibrary{intersections, pgfplots.fillbetween}
\pgfdeclarelayer{pre main}

\usepackage[colorlinks=true,linkcolor=black,urlcolor=black,citecolor=black,pdftitle={An AAP principle for Schr\"{o}dinger operators}, pdfauthor={S. Buccheri, L. Orsina and A.C. Ponce}]{hyperref}

\usepackage[abbrev]{amsrefs}
\usepackage{mathtools}
\usepackage{esint}

\usepackage{booktabs}
\usepackage[USenglish]{babel}
\usepackage{cancel}
\usepackage{times}
\usepackage{mathrsfs}
\usepackage{multirow,multicol}
\usepackage{amsfonts}
\usepackage{amsthm}
\usepackage{amsmath}
\makeatletter
\newdimen\rh@wd
\newdimen\rh@hta
\newdimen\rh@htb
\newbox\rh@box
\def\rh@measure#1{\setbox\rh@box=\hbox{$#1$}\rh@wd=\wd\rh@box \rh@hta=\ht\rh@box}

\def\widecheck#1{\rh@measure{#1}%
  \setbox\rh@box=\hbox{$\widehat{\vrule height \rh@hta width\z@ \kern\rh@wd}$}%
  \rh@htb=\ht\rh@box \advance\rh@htb\rh@hta \advance\rh@htb\p@
  \ooalign{$\vrule height \ht\rh@box width\z@ #1$\cr
           \raise\rh@htb\hbox{\scalebox{1}[-1]{\box\rh@box}}\cr}}
\makeatother

\usepackage{amsmath,amssymb}

\usepackage{enumerate}
\usepackage{constants}

\def\norma#1#2{\|#1\|_{\lower 4pt \hbox{$\scriptstyle #2$}}}

\def\finedim
\def\gw{G_{\tilde{k}}(w_n)}

\DeclareMathOperator{\supp}{supp}

 \usepackage{color}
 \numberwithin{equation}{section}
\usepackage{todonotes}

\theoremstyle{plain}
\newtheorem{proposition}{Proposition}[section]
\newtheorem{lemma}[proposition]{Lemma}
\newtheorem{theorem}[proposition]{Theorem}
\newtheorem{corollary}[proposition]{Corollary}
\newtheorem{definition}[proposition]{Definition}

\theoremstyle{definition}

\theoremstyle{remark}
\newtheorem{example}[proposition]{Example}

\newcommand{\abs}[1]{\mathopen\lvert#1\mathclose\rvert}

\newcommand{\norm}[1]{\mathopen\lVert#1\mathclose\rVert}

\newcommand{\ps}[2]{(#1\vert#2)}
\newcommand{\N}{{\mathbb N}}
\newcommand{\R}{{\mathbb R}}
\newcommand{\dif}{\,\mathrm{d}}
\DeclareMathOperator{\diam}{diam}
\DeclareMathOperator{\Div}{div}
\DeclareMathOperator{\capt}{cap}

\title[An AAP principle for Schr\"{o}dinger operators]{An Agmon-Allegretto-Piepenbrink principle for Schr\"{o}dinger operators}

\author{Stefano Buccheri}
\address{
Stefano Buccheri\hfill\break\indent
University of Vienna\hfill\break\indent
Faculty of Mathematics\hfill\break\indent
Oskar-Morgenstern-Platz 1\hfill\break\indent
1090 Vienna, Austria}

\author{Luigi Orsina}
\address{
Luigi Orsina\hfill\break\indent
``Sapienza'' Universit\`a di Roma\hfill\break\indent
Dipartimento di Matematica \hfill\break\indent
P.le A.~Moro 2\hfill\break\indent
00185 Roma, Italy}

\author{Augusto C. Ponce}
\address{
Augusto C. Ponce\hfill\break\indent
 Université catholique de Louvain\hfill\break\indent
 Institut de Recherche en Mathématique et Physique\hfill\break\indent
 Chemin du cyclotron 2, L7.01.02\hfill\break\indent
1348 Louvain-la-Neuve, Belgium}

\begin{document}

\dedicatory{
To Ildefonso Díaz, mathematician and gentleman.\\
\bigskip
``Non fuyades, cobardes y viles criaturas,\\
que un solo caballero es el que os acomete.''\\
(Cervantes, Don Quijote)
}

\begin{abstract}
We prove that each Borel function \(V : \Omega \to [{-\infty}, +\infty]\) defined on an open subset \(\Omega \subset \R^{N}\) induces a decomposition $\Omega = S \cup \bigcup_{i} D_{i}$ such that
every function in $W^{1,2}_{0}(\Omega) \cap L^{2}(\Omega; V^{+} \dif x)$ is zero almost everywhere on $S$ and existence of 
nonnegative supersolutions of $-\Delta + V$ on each component $D_{i}$ yields nonnegativity of the associated quadratic form
\(
\int_{D_{i}} (|\nabla \xi|^2+V\xi^2).
\)
\end{abstract}



\maketitle 

\section{Introduction}

Given a smooth bounded connected open set $\Omega\subset \mathbb{R}^{N}$ with \(N \ge 2\) and a Borel function  \(V : \Omega \to [{-\infty}, +\infty]\), we investigate the relation between existence of nonnegative solutions of the Dirichlet problem
\begin{equation}
		\label{eqDirichletProblemIntroduction}
	\left\{
	\begin{alignedat}{2}
	- \Delta u + Vu & = \mu	&& \quad \text{in \(\Omega\),}\\
	u & = 0 && \quad \text{on \(\partial\Omega\),}
	\end{alignedat}
	\right.
	\end{equation}
and nonnegativity of the associated quadratic form, namely
\begin{equation}\label{gap}
\int_{\Omega} (|\nabla \xi|^2+V\xi^2) \ge 0.
\end{equation}

When $V$ belongs to a Lebesgue space $L^{p}(\Omega)$ with $p > N/2$ or, more generally, to the Kato class $K(\Omega)$, see Definition~\ref{katoclass} below, the Agmon-Allegretto-Piepenbrink (AAP) principle provides a strong link between existence of nonnegative supersolutions and certain spectral properties of the Schr\"odinger operator \(-\Delta + V\); see e.g.~\cite{simon}. 
A typical formulation of the AAP principle for \eqref{eqDirichletProblemIntroduction} is the following

\begin{theorem}
	\label{theoremAAPCommon}
	If \(V \in K(\Omega)\), then the Dirichlet problem \eqref{eqDirichletProblemIntroduction} has a nontrivial nonnegative solution \(u \in W_{0}^{1, 2}(\Omega)\) for some nonnegative \(\mu \in L^{2}(\Omega)\) if and only if inequality \eqref{gap} holds for every \(\xi \in W_{0}^{1, 2}(\Omega)\).
\end{theorem}

Such a statement is known to specialists but in Section~\ref{appendix} we present a proof for the reader's convenience. 
A key ingredient is the compact imbedding \(W_{0}^{1, 2}(\Omega) \subset L^{2}(\Omega; V \dif x)\) that holds when \(V \in K(\Omega)\).{}

Various attempts have been made to further generalize the class of potentials \(V\) in the AAP principle; see for instance \cites{JMV,LSV,LP,PinPsa}.
However, an equivalence as in Theorem~\ref{theoremAAPCommon} for a general Borel function \(V\) is hopelessly false. 
For example, taking \(\Omega = B_{1}\) to be the unit ball centered at \(0\) and \(V(x) = 1/|x_{1}|^{\alpha}\) with \(1 \le \alpha < 2\), one has that \eqref{gap} is trivially satisfied by positivity of \(V\) but there exists no nonnegative function \(v\) that satisfies \(-\Delta v + V v \ge 0\) in the sense of distributions in \(\Omega\), other than the trivial one; see \cite{OrsinaPonce:2008}*{Theorem~1.4} and Example~\ref{exampleOrsinaPonce} below.

In contrast, existence of nonnegative supersolutions typically yields nonnegativity of the quadratic form under weaker assumptions on \(V\), as it has been proved in \cite{LP} for \(V \in L^{1}(\Omega)\). 
As we shall see, this property is true even for general Borel potentials \(V\).

One of the novelties of our work is to address this issue by adapting the concept of supersolution to the Schrödinger operator \(- \Delta + V\).
Our approach is based on the existence of a Borel set \(S \subset \Omega\) with the following properties:
\begin{enumerate}[(a)]
	\item{}
	\label{item-230}
	every solution \(u\) of \eqref{eqDirichletProblemIntroduction} satisfies \(u = 0\) almost everywhere in \(S\),
	\item{}
	\label{item-232} if moreover \(u\) is nonnegative, then \(\mu \le 0\) on \(S\), regardless of \(\mu\) on \(\Omega \setminus S\).
\end{enumerate}

The Borel set \(S\)  that we consider is related solely to the positive part \(V^{+}\) and has been introduced in \cite{OP} as the level set
\begin{equation}
\label{eqDefintionS}
S \vcentcolon= \{\widehat{\zeta_{1}} = 0\},
\end{equation}
where \(\widehat{\zeta_{1}}\) denotes the precise representative of the \emph{torsion function} \(\zeta_{1}\)\,, which minimizes the energy functional
\[{}
\xi \in  W_{0}^{1, 2}(\Omega) \cap L^{2}(\Omega; V^{+} \dif x)
\longmapsto \frac{1}{2} \int_{\Omega} (|\nabla \xi|^{2} + V^{+}\xi^{2}) - \int_{\Omega} \xi.
\]
One verifies that \(\widehat{\zeta_{1}}\) exists at every point of \(\Omega\) since \(\zeta_{1}\) can be written as the difference between a continuous and a bounded subharmonic function; see e.g.~\cite{OrsinaPonce:2017}*{Proposition~8.1}.
This set \(S\) has been used in \cite{OP} to investigate the failure of the strong maximum principle for the Schrödinger operator \(- \Delta + V^{+}\) and non-existence of solutions of the associated Dirichlet problem.{}
That \(S\) defined by \eqref{eqDefintionS} satisfies properties \eqref{item-230} and \eqref{item-232} is justified in Section~\ref{sectionDuality} below.
An alternative way of identifying \(S\) without explicitly relying on the torsion function is also available using the Wiener criterion of Dal~Maso and Mosco~\cite{DalMasoMosco}; see \cite{PonceWilmet}.

In a classical situation where \(V^{+} \in K(\Omega)\), the strong maximum principle holds and then one has \(S = \emptyset\). 
However, for general potentials \(V^+\) one should expect to have \(S \neq \emptyset\).{}
For example, if \(a \in \Omega\) and \(V(x) = 1/|x - a|^{2}\), then \(S = \{a\}\).{}
The set \(S\) may be larger than a finite union of singletons: Take a smooth open set \(\omega \Subset \Omega\) and \(V(x) = 1/d(x, \partial\omega)^{2}\), then \(S = \partial\omega\).
Choosing instead \(V(x) = 1/d(x, \omega)^{2}\), one has \(S = \overline\omega\).{}
These two examples are physically related to the localized effect of particles in Quantum Mechanics that has been beautifully addressed in recent years by I.~Díaz~\cites{diaz1,diaz}.

From \cite{OP}, we know that \(\Omega \setminus S\) can be written as a disjoint union of its Sobolev-connected components, namely
\begin{equation}
\label{eqDecomposition}
\Omega \setminus S = \bigcup_{i \in I}{D_{i}}\,,
\end{equation}
where the index set \(I\) is finite or countably infinity.
In the terminology introduced in \cite{OP}*{Sections~10 and~11}, each set \(D_{i}\) is Sobolev-open in the sense that there exists \(v_{i} \in W_{0}^{1, 2}(\Omega)\) whose precise representative \(\widehat{v_{i}}\) is defined everywhere in \(\Omega\) and \(D_{i} = \{\widehat{v_{i}} > 0\}\).
In addition, \(D_{i}\) is Sobolev-connected meaning that it cannot be further written as a disjoint finite reunion of Sobolev-open subsets.
When \(S\) is closed in \(\Omega\) with respect to the Euclidean topology, the sets \(D_{i}\) coincide with the usual connected components of \(\Omega \setminus S\).

Decomposition \eqref{eqDecomposition} propagates to every \(\xi \in  W_{0}^{1, 2}(\Omega) \cap L^{2}(\Omega; V^{+} \dif x)\) as we prove in Section~\ref{sectionCutoff} that \(\xi = 0\) almost everywhere in \(S\) and, for every \(i \in I\),
\(\xi\chi_{D_{i}} \in W_{0}^{1, 2}(\Omega)\)\,;
see Proposition~\ref{propositionSobolevCutoff}.{}
Thus, \(\xi\) can be written as a sum of Sobolev functions,
\begin{equation}
\label{eqDecompositionFunction}
\xi = \sum_{i \in I}{\xi \chi_{D_{i}}}
\quad \text{almost everywhere in \(\Omega\).}
\end{equation}

In the spirit of the decompositions \eqref{eqDecomposition} and \eqref{eqDecompositionFunction}, we localize the APP principle on each component $D_i$\,: 

\begin{theorem}
	\label{theoremPoincare}
	Let \(i \in I\).
	If \eqref{eqDirichletProblemIntroduction} has a nonnegative distributional solution \(u\) for a finite Borel measure \(\mu\) in \(\Omega\) such that \(\mu \ge 0\) in \(D_{i}\) and  \(\int_{D_{i}} u > 0\), then
	\begin{equation}
	\label{eqPoincareWeak}
	\int_{D_{i}} (|\nabla\xi|^{2} + V\xi^{2}) 
	\ge 0
	\quad \text{for every \(\xi \in W_{0}^{1, 2}(\Omega) \cap L^{2}(\Omega; V^{+} \dif x)\).} 
	\end{equation}
\end{theorem}

We recall that $u$ is a distributional solution of \eqref{eqDirichletProblemIntroduction} whenever $u \in W^{1, 1}_0(\Omega) \cap L^1(\Omega; V \dif x)$ and 
\begin{equation}
\label{eq-379}
-\Delta u + V u = \mu 
\quad \text{in the sense of distributions in \(\Omega\).}
\end{equation}
We show in Sections~\ref{sectionCutoff} and~\ref{sectionDuality} that this equation can be localized in \(D_{i}\) so that \(u\chi_{D_{i}}\) satisfies
\[
-\Delta (u\chi_{D_{i}}) + V u\chi_{D_{i}} = \mu\lfloor_{D_{i}}{} - \tau
\quad \text{in the sense of distributions in \(\Omega\),}
\]
where the measure \(\mu\lfloor_{D_{i}}\) is defined by \(\mu\lfloor_{D_{i}}(A) = \mu(D_{i} \cap A)\) for every Borel set \(A \subset \Omega\) and $\tau$ is a measure carried by $S$, that is $|\tau|(\Omega \setminus S) = 0$.
Under the assumptions of Theorem~\ref{theoremPoincare}, we have that $\tau$ is nonnegative and we may apply the strong maximum principle for \(-\Delta + V^{+}\) in \(D_{i}\) to deduce that
\[{}
\liminf_{r \to 0}{\fint_{B_{r}(x)} u} > 0
\quad \text{for every \(x \in D_{i}\)\,,}
\]
where \(\fint_{B_{r}(x)}\) is the average integral on the ball \(B_{r}(x)\).

Theorem~\ref{theoremPoincare} thus provides one with a Poincaré inequality that gives the imbedding
\begin{equation*}
	W_{0}^{1, 2}(\Omega) \cap L^{2}(\Omega; V^{+} \dif x) \subset L^{2}(D_{i}; V^{-} \dif x),
\end{equation*}
which can be seen as a compatibility condition between $V^+$ and $V^-$ to have existence of supersolutions in \(D_{i}\)\,.{}
When \eqref{eqPoincareWeak} is satisfied, we consider the vector space
\[
H_i(\Omega) \vcentcolon= \Bigl\{ \xi\chi_{D_{i}} \in W_{0}^{1, 2}(\Omega) : \xi \in W^{1, 2}_0(\Omega)\cap  L^2(\Omega; V^+ \dif x) \Bigr\}
\]
equipped with the seminorm
\[{}
\|\xi\|_{i}
\vcentcolon= \biggl[\int_{D_{i}} (|\nabla\xi|^{2} + V\xi^{2}) \biggr]^{\frac{1}{2}}.
\]
We then denote by \(\mathcal{H}_{i}(\Omega)\) the completion of \(H_{i}(\Omega)\).

In the spirit of the concept of weighted spectral gap by Pinchover and Tintarev~\cite{PinTin},
it is natural to investigate whether there is room for an improvement of \eqref{eqPoincareWeak} in the form
\begin{equation}
	\label{eqPoincareStrong}
	\int_{D_{i}} (|\nabla \xi|^2+V\xi^2) \ge \int_{D_{i}}  w_{i} \xi^2{}
	\quad \text{for every \(\xi \in W_{0}^{1, 2}(\Omega) \cap L^{2}(\Omega; V^{+} \dif x)\)},
\end{equation}
for some measurable weight \(w_{i} : D_{i} \to (0, +\infty)\).
In contrast with \cite{PinTin}, we do not require \(w_{i}\) to be continuous.
An inequality of the type \eqref{eqPoincareStrong} is typically used to prove decay of solutions associated to the operator $-\Delta + V$ and long-time behavior of the $L^p$ norm for the Schr\"odinger semigroup in unbounded domains; see \cites{Agmon, simonbis, Pinchover}.
For the critical Hardy potential \(V = -\kappa/|x|^{2}\), estimate \eqref{eqPoincareStrong} holds with \(w_{i}\) constant and such an inequality has been used to investigate nonlinear eigenvalue problems and asymptotics of the heat equation; see \cites{BrezisVazquez,VazquezZuazua,DupaigneNedev,DavilaDupaigne}.

\begin{theorem}
	\label{theoremPoincare-bisA}
	Let \(i \in I\).{}
	If \eqref{eqDirichletProblemIntroduction} has a nonnegative distributional solution for a finite Borel measure \(\mu\) in \(\Omega\) such that \(\mu \ge 0\) in \(D_{i}\) and \(\mu(D_{i}) > 0\), then there exists a measurable function \(w_{i} : D_{i} \to (0, +\infty)\) such that \eqref{eqPoincareStrong} holds.
\end{theorem}

We prove Theorem~\ref{theoremPoincare-bisA} in Section~\ref{sectiontheoremPoincare-bisA}.
The possibility that the measure in the statement of Theorem~\ref{theoremPoincare} satisfies \(\mu(D_{i}) = 0\) does not exclude the existence of another one with \(\widetilde\mu \ge 0\) in \(D_{i}\) and \(\widetilde\mu(D_{i}) > 0\), for which a nonnegative solution exists and then the stronger conclusion of Theorem~\ref{theoremPoincare-bisA} holds; see Example~\ref{exampleSupersolution}. 
This phenomenon is due to the fact that distributional solutions need not belong to \(L^{2}(\Omega; V^{-} \dif x)\), see Corollary~\ref{propositionSupersolutionZero}, and is in sharp contrast to what happens with the operator \(-\Delta - \lambda_{1}\),{}
where \(\lambda_{1}\) is the first eigenvalue of the Laplacian.
In such a case, one has \(S = \emptyset\) and if \(u\) solves
\[{}
	\left\{
	\begin{alignedat}{2}
	- \Delta u - \lambda_{1}u & = \mu	&& \quad \text{in \(\Omega\),}\\
	u & = 0 && \quad \text{on \(\partial\Omega\),}
	\end{alignedat}
	\right.{}
\]
for some finite nonnegative measure \(\mu\), then using the first eigenfunction \(\varphi\) of \(-\Delta\) as test function one deduces that necessarily \(\mu = 0\) and \(u = \alpha\varphi\) with \(\alpha \in \R\).

Under the stronger assumption \eqref{eqPoincareStrong},
any element $\xi \in \mathcal{H}_i(\Omega)$ can be naturally associated to a function in $L^2(D_i; w_i \dif x)$
and in particular is well-defined almost everywhere in \(D_{i}\)\,.{}
Moreover, using the direct method of the Calculus of Variations one easily deduces that, for every \(h \in L^{2}(D_{i}; w_{i} \dif x)\), there exists a unique minimizer \(\theta_{i, h}\) of the energy functional
\begin{equation}
	\label{eqFunctional}
E(\xi) = \frac{1}{2} \|\xi\|_{i}^{2} - \int_{D_{i}} w_{i}h\xi{}
\quad \text{with \(\xi \in \mathcal{H}_{i}(\Omega)\).}
\end{equation}
While  \(\theta_{i, h}\) satisfies the Euler-Lagrange equation in \(\mathcal{H}_{i}(\Omega)\), it need not be true that \(\theta_{i, h}\) verifies \eqref{eqDirichletProblemIntroduction} in the sense of distributions for some finite measure \(\mu\); see Example~\ref{exampleFailureDistributions}.
We prove nevertheless in Section~\ref{sectionlast} the following

\begin{theorem}
	\label{theoremPoincare-bisB}
	Let \(i \in I\).{}
	If \eqref{eqDirichletProblemIntroduction} has a nonnegative distributional solution \(u\) for a finite Borel measure \(\mu\) in \(\Omega\) such that \(\mu \ge 0\) in \(D_{i}\) and \(\mu(D_{i}) > 0\), then there exists a bounded measurable function \(w_{i} : D_{i} \to (0, +\infty)\) satisfying \eqref{eqPoincareStrong} and such that, for every \(h \in L^{2}(D_{i}; w_{i} \dif x)\) with \(0 \le h \le u\), 
	the minimizer of \eqref{eqFunctional} is a nonnegative distributional solution of \eqref{eqDirichletProblemIntroduction} for a finite Borel measure \(\mu_{i,h}\) in $\Omega$ that verifies
		\[{}
		\mu_{i, h} = w_{i} h \dif x \quad \text{on $D_{i}$\,.}
		\]
\end{theorem}

An interesting open problem is to find a weight $w_{i}$ such that the conclusion of 
Theorem~\ref{theoremPoincare-bisB} holds for every nonnegative $h \in L^{2}(D_{i}; w_{i} \dif x)$.


\section{Proof of Theorem~\ref{theoremAAPCommon}}\label{appendix}

In this section we provide a proof of Theorem~\ref{theoremAAPCommon}. 
We assume that \(N \ge 3\)\,; the case \(N = 2\) requires some adaptation as one needs to modify the definition of the Kato class using the logarithm in replacement of the potential \(1/|z|^{N-2}\).
Let us start recalling the definition of the Kato class $K(\Omega)$.

\begin{definition}\label{katoclass}
A Borel function $V:\Omega\to \mathbb [{-\infty}, +\infty]$ belongs to the Kato class $K(\Omega)$ whenever
\[
\lim_{r\to 0}{\sup_{x \in \Omega} \int_{B_{r}(x)\cap\Omega}\frac{|V(y)|}{|x-y|^{N-2}} \dif y} = 0.
\]
\end{definition}

The proof of Theorem~\ref{theoremAAPCommon} is based on the following lemma:

\begin{lemma}\label{morreyadams}
If\/ $V\in K(\Omega)$, then for each $\epsilon>0$ there exists $C > 0$ such that
\[
\int_{\Omega} |V|\varphi^2{}
\le \epsilon \int_{\Omega}|\nabla \varphi|^2 + C \int_{\Omega}\varphi^2 \quad \text{for every \(\varphi\in C^\infty_c(\Omega)\).}
\]
\end{lemma}
To our knowledge, the proof of this lemma goes back to \cite{schechter}*{Chapter~6, Theorem~9.3}. 
We provide a more direct proof based on \cite{zamboni}. 
The estimate above implies that \(W_{0}^{1, 2}(\Omega) \subset L^{2}(\Omega; V \dif x)\) and that the inclusion is compact.
Indeed, by density of \(C_{c}^{\infty}(\Omega)\) in \(W_{0}^{1, 2}(\Omega)\), the inequality holds with \(\varphi = \xi \in W_{0}^{1, 2}(\Omega)\).
Moreover, if \((\xi_{n})_{n \in \N}\) is a bounded sequence in \(W_{0}^{1, 2}(\Omega)\), then by the Rellich-Kondrashov Compactness Theorem we may extract a subsequence such that \(\xi_{n_{j}} \to u\) in \(L^{2}(\Omega)\) for some \(u \in W_{0}^{1, 2}(\Omega)\).{}
Using the inequality above for \(\xi_{n_{j}} - u\) we then have, for every \(\epsilon > 0\),{}
\[{}
\limsup_{j \to \infty}{\int_{\Omega} |V|(\xi_{n_{j}} - u)^2}
\le \epsilon \limsup_{j \to \infty}\int_{\Omega}|\nabla (\xi_{n_{j}} - u)|^2.
\]
Since \(\epsilon\) is arbitrary, the convergence \(\xi_{n_{j}} \to u\) in \(L^{2}(\Omega; V \dif x)\) then follows.

\resetconstant
\begin{proof}[Proof of Lemma~\ref{morreyadams}]
We first show that, for every \(\varphi\in C_c^\infty(\Omega)\),
\begin{equation}
	\label{eq-493}
\int_{\Omega}|V| \varphi^2{}
\le \biggl( \sup_{x \in \Omega} \int_{\supp{\varphi}} \frac{|V(y)|}{|x-y|^{N-2}} \dif y  \biggr) \int_{\Omega}|\nabla \varphi|^2.
\end{equation}
This inequality is proved in \cite{schechter}*{Chapter~6, Theorem~8.8} using Bessel potentials, but here we follow a simpler idea from \cite{zamboni}: 
Since
\[
|\varphi(x)|{}
\le \Cl{eq-507} \int_{\Omega}\frac{|\nabla \varphi(y)|}{|x-y|^{N-1}} \dif y,
\]
we get
\begin{equation}\label{19:31}
\begin{split}
\int_{\Omega}|V(x)| |\varphi(x)|^{2} \dif x 
& \le \Cr{eq-507} \int_{\Omega}\left(\int_{\Omega}\frac{|V(x)| |\varphi(x)|}{|x-y|^{N-1}} \dif x\right) |\nabla \varphi(y)| \dif y\\
& \le \Cr{eq-507} \biggl(\int_{\Omega}\left(\int_{\Omega} \frac{|V(x)| |\varphi(x)|}{|x-y|^{N-1}} \dif x \right)^2 \dif y\biggr)^{\frac12} \left(\int_{\Omega}|\nabla \varphi (y)|^2 \dif y\right)^{\frac12}.
\end{split}
\end{equation}
Note that
\[{}
\left(\int_{\Omega} \frac{|V(x)| |\varphi(x)|}{|x-y|^{N-1}} \dif x \right)^2{}
\le \left(\int_{\supp{\varphi}} \frac{|V(z)|}{|z-y|^{N-1}} \dif z\right) \left(\int_{\Omega}\frac{|V(x)| |\varphi(x)|^2}{|x-y|^{N-1}} \dif x\right){}
\]
and
\[{}
\int_{\Omega} \frac{\dif y}{|z-y|^{N-1}|x-y|^{N-1}}
\le \frac{\Cl{cte-526}}{|z-x|^{N-2}}.
\]
From both estimates and Fubini's theorem, we get 
\[
\begin{split}
\int_{\Omega}\biggl(\int_{\Omega}\frac{|V(x)| |\varphi(x)|}{|x-y|^{N-1}} \dif x \biggr)^2 \dif y
& \le \Cr{cte-526} \int_{\supp{\varphi}}\int_{\Omega}   \frac{|V(z)||V(x)| |\varphi(x)|^{2}}{|z-x|^{N-2}} \dif x \dif z\\
& \le \Cr{cte-526} \biggl(\sup_{x\in\Omega}\int_{\supp{\varphi}} \frac{|V(z)|}{|z-x|^{N-2}} \dif z \biggr) \int_{\Omega} |V(x)||\varphi(x)|^2 \dif x.
\end{split}
\]
Combining this estimate with \eqref{19:31}, we obtain \eqref{eq-493}.

To conclude, we cover $\Omega$ with a regular net of balls $B_{{\delta}}(x_i)$ with $i=1,\dots, M$, where the number of overlaps depends solely on the dimension \(N\).
Let $\psi_1, \dots, \psi_M$ be a smooth partition of unity subordinated to this covering. 
Given \(\varphi \in C_{c}^{\infty}(\Omega)\), we have
\begin{equation}
\label{eq533}
\int_{\Omega} |V|\varphi^2{}
\le \Cl{cte-546} \sum_{i=1}^M \int_{\Omega} |V|(\varphi\psi_i)^2,
\end{equation}
for some constant \(\Cr{cte-546} > 0\) depending on \(N\).
By \eqref{eq-493} applied to \(\varphi\psi_{i}\), we get
\begin{equation}
\label{eq540}
\int_{\Omega}|V| (\varphi\psi_{i})^2
\le \eta(\delta) \int_{\Omega}|\nabla (\varphi\psi_{i})|^2,
\end{equation}
with 
\[
\eta(\delta){}
\vcentcolon= 
\sup_{x \in \Omega} \int_{\Omega \cap B_{\delta}(x)} \frac{|V(y)|}{|x-y|^{N-2}} \dif y.
\]
We have \(0 \le \psi_{i} \le 1\) and we may assume that \(|\nabla\psi_{i}| \le \C/\delta\) for every \(i\).
Combining \eqref{eq533} and \eqref{eq540}, we get
\[
\int_{\Omega} V\varphi^2{}
\le \Cl{eq-568}\eta(\delta) \biggl(\int_{\Omega}|\nabla \varphi|^2 + \frac{1}{\delta^{2}} \int_{\Omega}\varphi^2\biggr),
\]
where \(\Cr{eq-568} > 0\) depends on \(N\).
Since \(V \in K(\Omega)\), we have \(\eta(\delta) \to 0\) as \(\delta \to 0\).{}
Thus, given \(\epsilon > 0\), we may take \(\delta > 0\) small enough so that \(\Cr{eq-568}\eta(\delta) \le \epsilon\).
\end{proof}

\begin{proof}[Proof of Theorem \ref{theoremAAPCommon} ``\(\Longleftarrow\)'']  
We prove the reverse implication, so we assume that
\begin{equation}\label{11:28}
\int_{\Omega}(|\nabla \xi|^2 + V\xi^2){}
\ge 0 \quad \text{for every \(\xi\in W^{1,2}_0(\Omega)\).}
\end{equation}
Our goal is to prove that the minimization problem
\begin{equation}\label{11:39}
\lambda_1{}
=\inf\left\{ \int_{\Omega}(|\nabla \xi|^2+V\xi^2) \ : \ \xi \in W^{1,2}_0(\Omega) \ \mbox{ and } \ \int_{\Omega} \xi^{2} =1\right\}
\end{equation}
achieves its infimum for some nonnegative function \(u\).
Note that, thanks to \eqref{11:28}, the infimum $\lambda_1$ is well-defined and nonnegative. 
Once one knows that the minimum exists, it is standard to show that
\[
-\Delta u + Vu = \lambda_1 u
\quad \text{in the sense of distributions in \(\Omega\).}
\] 

Take a minimizing sequence \((\xi_{n})_{n \in \N}\) of \eqref{11:39}.{}
We show that \((\nabla\xi_{n})_{n \in \N}\) is bounded in \(L^{2}(\Omega; \R^{N})\).{}
To this end, for \(n\) large enough, we apply Lemma~\ref{morreyadams} with \(\epsilon = 1/2\) to get 
\[
\int_{\Omega}|\nabla \xi_n|^2{}
\le (\lambda_1 + 1) - \int_{\Omega}V\xi_n^2{}
\le (\lambda_1 + 1) + \frac{1}{2} \int_{\Omega}|\nabla \xi_n|^2 + C \int_{\Omega}\xi_n^2.
\]
Recalling the normalization condition of $\xi_n$\,, it follows that
\[
\frac{1}{2} \int_{\Omega}|\nabla \xi_n|^2 \le (\lambda_1 + 1) + C.
\]

By boundedness of \((\xi_{n})_{n \in \N}\) in \(W_{0}^{1, 2}(\Omega)\) and compactness of the inclusion \(W_{0}^{1, 2}(\Omega) \subset L^{2}(\Omega; V \dif x)\), we may extract a subsequence \((\xi_{n_{j}})_{j \in \N}\) that converges weakly to some $v$ in $W^{1, 2}_0(\Omega)$ and also strongly in $L^2(\Omega)$ and \(L^{2}(\Omega; V \dif x)\). 
We thus have
\[
\int_{\Omega}(|\nabla v|^2 + Vv^2)
\le \liminf_{j \to\infty}\int_{\Omega}(|\nabla \xi_{n_{j}}|^2 + V\xi_{n_{j}}^2)=\lambda_1,
\]
with $\int_{\Omega} v^{2} =1$.
We deduce that \(v\) is a minimizer of \eqref{11:39}.
Hence, the nonnegative function \(|v|\) is also a minimizer and we have the conclusion by taking \(u = |v|\) and \(\mu = \lambda_{1}u\).
\end{proof}

To prove the direct implication of Theorem~\ref{theoremAAPCommon}, we rely on the following lemma that is a typical variational setting where existence of a positive supersolution yields a weighted Poincaré inequality and is valid for a general Borel function \(V\).

\begin{lemma}
\label{lemmaPoincareVariational}
	Let \(f \in L^{2}(\Omega)\) and \(u \in W_{0}^{1, 2}(\Omega) \cap L^{2}(\Omega; V \dif x)\) be nonnegative functions such that, for every nonnegative \(\xi \in W_{0}^{1, 2}(\Omega) \cap L^{2}(\Omega; V \dif x)\),
	\begin{equation}
		\label{eq631}
	\int_{\Omega}(\nabla u \cdot \nabla\xi + V u \xi){}
	\ge \int_{\Omega} f \xi.
	\end{equation}
	Then, for every \(\xi \in W_{0}^{1, 2}(\Omega) \cap L^{2}(\Omega; V \dif x)\), we have \(f\xi^{2} = 0\) almost everywhere in \(\{u = 0\}\) and
	\begin{equation}
		\label{eq637}
	\int_{\{u > 0\}} (\abs{\nabla \xi}^{2} + V \xi^{2})
	\ge \int_{\{u > 0\}}  \frac{f}{u} \xi^{2}.
	\end{equation}
\end{lemma}

\begin{proof}[Proof of Lemma~\ref{lemmaPoincareVariational}]
	Given \(\xi \in W_{0}^{1, 2}(\Omega) \cap L^{2}(\Omega; V \dif x) \cap L^{\infty}(\Omega)\) and \(\epsilon > 0\), we have \(\eta \vcentcolon= \xi^2/(u + \epsilon) \in W_{0}^{1, 2}(\Omega) \cap L^{2}(\Omega; V \dif x)\) and \(\eta\) is nonnegative.{}
	Using \(\eta\) as a test function in \eqref{eq631}, we get
	\[{}
	2\int_{\Omega} \frac{\nabla u \cdot \nabla \xi}{u + \epsilon} \xi
	- \int_{\Omega} \frac{|\nabla u|^2 }{(u + \epsilon)^2}\xi^2
	+ \int_{\Omega}\frac{V u }{u + \epsilon}\xi^2
	\ge \int_{\Omega} \frac{f}{u + \epsilon} \xi^{2}.  
	\]
	Since \(\nabla u = 0\) almost everywhere in \(\{u = 0\}\),
	\[{}
	2\frac{\nabla u \cdot \nabla \xi}{u + \epsilon} \xi{}
	\le \frac{|\nabla u|^2 }{(u + \epsilon)^2}\xi^2 + |\nabla \xi|^{2} \chi_{\{u > 0\}}.
	\]
	Thus,
	\begin{equation}
	\label{eq581}
	\int_{\{u > 0\}} \abs{\nabla \xi}^{2} + 
	\int_{\Omega} \frac{V u }{u + \epsilon}\xi^2
	\ge \int_{\Omega} \frac{f}{u + \epsilon} \xi^2.
	\end{equation}
	Note that, as \(\epsilon \to 0\),
	\[{}
	\frac{V u }{u + \epsilon}\xi^2 
	\to V\xi^2 \chi_{\{u > 0\}}
	\quad \text{and} \quad{}
	\frac{f}{u + \epsilon}\xi^2 \to \frac{f}{u}\xi^2 \chi_{\{u > 0\}} +\infty \,\chi_{\{f\xi^{2} > 0\} \cap \{u = 0\}}.
	\]
	Since
	\[{}
	\Bigl| \frac{V u }{u + \epsilon}\xi^2 \Bigr|{}
	\le |V|\xi^{2} \in L^{1}(\Omega)
	\]
	it follows from Fatou's lemma and \eqref{eq581} that \(\{f\xi^{2} > 0\} \cap \{u = 0\}\) is negligible with respect to the Lebesgue measure.
	Applying  the Dominated Convergence Theorem, we deduce estimate \eqref{eq637} under the additional assumption that \(\xi \in L^{\infty}(\Omega)\).{}
	To get the estimate for a general \(\xi \in W_{0}^{1, 2}(\Omega) \cap L^{2}(\Omega; V \dif x)\), it suffices to apply \eqref{eq637} for \(T_{k}(\xi)\) with \(k \ge 0\) and let \(k \to \infty\).
\end{proof}

\begin{proof}[Proof of Theorem~\ref{theoremAAPCommon} ``\(\Longrightarrow\)'']
	We now assume that \(u \in W_{0}^{1, 2}(\Omega)\) is a nonnegative function that satisfies
	\[{}
	-\Delta u + Vu = \mu{}
	\quad \text{in the sense of distributions in \(\Omega\),}
	\]
	where \(\mu \in L^{2}(\Omega)\) is also nonnegative.
	Since \(V \in K(\Omega)\), by Lemma~\ref{morreyadams} we have \(W_{0}^{1, 2}(\Omega) \subset L^{2}(\Omega; V \dif x) \).{}
	As \(\mu \in L^{2}(\Omega)\), it then follows from the equation satisfied by \(u\) and the density of \(C_{c}^{\infty}(\Omega)\) in \(W_{0}^{1, 2}(\Omega)\),{}
	\[{}
	\int_{\Omega} (\nabla u \cdot \nabla\xi + Vu\xi{})
	= \int_{\Omega} \mu\xi{}
	\quad \text{for every \(\xi \in W_{0}^{1, 2}(\Omega)\).}
	\] 
	Thus, by Lemma~\ref{lemmaPoincareVariational},
	\begin{equation}
	\label{eq671}
		\int_{\{u > 0\}} (\abs{\nabla \xi}^{2} + V \xi^{2})
	\ge \int_{\{u > 0\}}  \frac{\mu}{u} \xi^{2}
	\ge 0
	\quad \text{for every \(\xi \in W_{0}^{1, 2}(\Omega)\).}
	\end{equation}
	Since \(V \in L^{1}(\Omega)\) and \(u\) is a nontrivial nonnegative function such that
	\[{}
	\int_{\Omega} (\nabla u \cdot \nabla\xi + V^{+}u\xi{})
	\ge 0
	\quad \text{for every nonnegative \(\xi \in W_{0}^{1, 2}(\Omega)\),}
	\] 
	by the strong maximum principle, see \cite{bookponce}*{Proposition~22.2}, we have \(u > 0\) almost everywhere in \(\Omega\) and then \eqref{eq671} gives
	\[{}
	\int_{\Omega} (\abs{\nabla \xi}^{2} + V \xi^{2})
	\ge 0
	\quad \text{for every \(\xi \in W_{0}^{1, 2}(\Omega)\).}
	\qedhere
	\]
\end{proof}


\section{Duality solutions for nonnegative potentials}
\label{sectionDualityPositive}

For each \(f \in L^{\infty}(\Omega)\), let \(\zeta_{f}\){}
be the unique minimizer of the energy functional
\begin{equation}
\label{eq731}
\xi \in  W^{1, 2}_0(\Omega)\cap L^2(\Omega; V^+ \dif x) \longmapsto \frac12 \int_{\Omega}(|\nabla \xi|^2+ V^+ \xi^2) - \int_{\Omega}f \xi.
\end{equation}
Denoting by \(\mathcal{M}(\Omega)\) the vector space of finite Borel measures in \(\Omega\), we recall the following notion from  \cite{Malusa_Orsina:1996}:

\begin{definition}
	Given \(\nu \in \mathcal{M}(\Omega)\), a function \(u \in L^{1}(\Omega)\) is a \emph{duality solution} of
\begin{equation}
	\label{eqDualitySolutionV+}
	\left\{
	\begin{alignedat}{2}
	-\Delta  u+ V^{+} u & = \nu && \quad \text{in } \Omega,\\
	 u & = 0 && \quad  \text{on }  \partial \Omega,
	\end{alignedat}
	\right.
\end{equation}
whenever 
\begin{equation}
	\label{eq753}
\int_{\Omega}u f 
= \int_{\Omega} \widehat{\zeta_f} \dif\nu  \quad \text{for every } f \in L^{\infty}(\Omega).
\end{equation}
\end{definition}

One verifies that the precise representative $\widehat{\zeta_f}$ is defined at every point in \(\Omega\), 
that is for every \(x \in \Omega\) there exists \(\widehat{\zeta_{f}}(x) \in \R\) such that
\[{}
\lim_{r \to 0}{\fint_{B_{r}(x)}|\zeta_{f} -\widehat{\zeta_{f}}(x)|}
= 0.
\]
Hence, the integral in the right-hand side of \eqref{eq753} is well-defined.

Every distributional solution is a duality solution, with the same datum.
The converse however need not be true; see Example~\ref{exampleOrsinaPonce}.
A major advantage of dealing with duality solutions is that, in contrast with distributional solutions, they always exist for every \(\nu \in \mathcal{M}(\Omega)\).
For example, if \(g \in L^{\infty}(\Omega)\), then a combination of the Euler-Lagrange equations satisfied by \(\zeta_{f}\) and \(\zeta_{g}\) implies that
\begin{equation}
\label{eq675}
\int_{\Omega} \zeta_{g} f
=
\int_{\Omega} \zeta_{f}\, g
\quad \text{for every } f \in L^{\infty}(\Omega).
\end{equation}
Since $\zeta_{f} = \widehat{\zeta_{f}}$ almost everywhere in \(\Omega\) by Lebesgue's differentiation theorem,
we deduce that $\zeta_{g}$ is the duality solution of \eqref{eqDualitySolutionV+} with datum $\nu = g \dif x$.
Moreover, from the estimate
\begin{equation*}
|\zeta_{g} | \leq \| g \|_{L^{\infty}(\Omega)} \zeta_{1}\,,{}
\end{equation*}
for every \(g \in L^{\infty}(\Omega)\) we then have
\begin{equation}
	\label{eq789}
\widehat{\zeta_{g}} = 0
\quad \text{in \(S\).}
\end{equation}

We investigate in this section a counterpart of \eqref{eq789} that is valid for any duality solution of \eqref{eqDualitySolutionV+}.
We begin with the following general property:

\begin{proposition}
	\label{propositionSZeroae}
	If \(u\) is a duality solution of \eqref{eqDualitySolutionV+} with datum \(\nu \in \mathcal{M}(\Omega)\), then
	\(u = 0\) almost everywhere in \(S\).
\end{proposition}

As we have pointed out, Proposition~\ref{propositionSZeroae} is straightforward when $u = \zeta_{g}$ with $g \in L^{\infty}(\Omega)$.
To prove the proposition in full generality, we use the property that, for every \(\nu \in \mathcal{M}(\Omega)\), the duality solution \(u\) of \eqref{eqDualitySolutionV+}
satisfies the representation formula
\begin{equation}
\label{eqRepresentationFormula}
u(x) 
= \int_{\Omega} \widehat{G_x} \dif\nu
\quad{}
\text{for almost every \(x \in \Omega\)},
\end{equation}
where $G_x$ denotes the duality solution of 
\[{}
	\left\{
	\begin{alignedat}{2}
	-\Delta G_{x}+ V^{+} G_{x} & = \delta_{x} && \quad \text{in } \Omega,\\
	 G_{x} & = 0 && \quad  \text{on }  \partial \Omega.
	\end{alignedat}
	\right.
\]
Formula \eqref{eqRepresentationFormula} is proved in \cite{OP}*{Lemma~13.2} when \(\nu\) is nonnegative.
For a general signed measure \(\nu\), one first applies \eqref{eqRepresentationFormula} to the duality solutions of \eqref{eqDualitySolutionV+} with nonnegative data \(\nu^{+}\) and \(\nu^{-}\).{}
We recall that this step can be performed since duality solutions exist for any data in \(\mathcal{M}(\Omega)\).
It then follows from \eqref{eqRepresentationFormula} applied to \(\nu^{+}\) and \(\nu^{-}\) that
\[{}
\int_{\Omega} \widehat{G_x} \dif|\nu|
< +\infty
\quad{}
\text{for almost every \(x \in \Omega\)}
\]
and, taking differences, we deduce \eqref{eqRepresentationFormula} for \(\nu = \nu^{+} - \nu^{-}\).

\begin{proof}[Proof of Proposition~\ref{propositionSZeroae}]
Since $G_{x}$ is a duality solution, using $\zeta_{1}$ as test function we have
\[
\int_{\Omega} G_{x} = \int_{\Omega} \widehat{\zeta_{1}} \dif \delta_{x} = \widehat{\zeta_{1}}(x)\,,
\]
for every $x \in \Omega$. Then, by nonnegativity of $G_{x}$ we deduce that 
\begin{equation}
	\label{eq826}
	\widehat{G_{x}} = 0
	\quad \text{for every \(x \in S\).}
\end{equation}
From \eqref{eqRepresentationFormula}, we deduce that $u(x) = 0$ for almost every $x \in S$. 
\end{proof}

We now prove a sharper version of Proposition~\ref{propositionSZeroae}, where the Lebesgue measure is replaced by the \(W^{1, 2}\) capacity.
More precisely, we recall that every duality solution \(u\) has a precise representative \(\widehat{u}\) which is defined quasi-everywhere, that is except in a set of \(W^{1, 2}\) capacity zero. This is a consequence of the fact that \(\Delta u\) is a finite measure in \(\Omega\); see \cite{OP}*{Proposition~4.1} and \cite{bookponce}*{Proposition~8.9}.
We then have the following statement whose proof follows along the lines of \cite{OP}*{Proposition~13.3}:

\begin{proposition}
	\label{propositionSZero}
	If \(u\) is a duality solution of \eqref{eqDualitySolutionV+} with datum \(\nu \in \mathcal{M}(\Omega)\), then
	\(\widehat{u} = 0\) quasi-everywhere in \(S\).
\end{proposition}

Let us denote by \(F\) the fundamental solution of \(-\Delta\) and focus our discussion on dimension \(N \ge 3\).{}
The starting observation is that if \(\nu \in \mathcal{M}(\Omega)\) is extended as \(0\) outside \(\Omega\), then for any smooth bounded open set \(U \supset \Omega\) the distributional solution of 
\[{}
	\left\{
	\begin{alignedat}{2}
	-\Delta w & = \nu && \quad \text{in } U,\\
	 w & = 0 && \quad  \text{on }  \partial U,
	\end{alignedat}
	\right.
\]
satisfies \(|w| \le F * |\nu|\) almost everywhere in \(U\).{}
Hence, if \(F * |\nu|\) is bounded, then so is  \(w\) and, by interpolation, we have \(w \in W_{0}^{1, 2}(U)\).{}
As a result, \(\nu\) belongs to the dual space \((W_{0}^{1, 2}(U))'\) and is \emph{diffuse} with respect to the \(W^{1, 2}\) capacity, that is, for any Borel set \(A \subset \Omega\) with \(W^{1, 2}\)~capacity zero, one has \(\tau(A) = 0\).{}
.{}
The action of \(\nu\) as an element in \((W_{0}^{1, 2}(U))'\) is given by
	\begin{equation}
		\label{eq822}
	\nu[v]
	\vcentcolon= \int_{U} \nabla v \cdot \nabla w
	= \int_{\Omega}{\widehat{v} \dif\nu}
	 \quad \text{for every $v \in W_{0}^{1,2}(U)$.}
	\end{equation}
In dimension \(N = 2\), the strategy above needs to be slightly adapted since \(F(x) \to -\infty\) as \(|x| \to \infty\) and one cannot expect to have \(F * |\nu|\) bounded in \(\R^{2}\).

\begin{proof}[Proof of Proposition~\ref{propositionSZero}]
	We first assume that the Newtonian potential \(F * |\nu|\) is bounded.
	In this case, by comparison with \(F  * |\nu|\) we have \(u \in L^{\infty}(\Omega)\) and then
	\(u \in W_{0}^{1, 2}(\Omega) \cap L^{2}(\Omega; V^{+} \dif x)\).{}
	Moreover, \(u\) satisfies the Euler-Lagrange equation
	\[{}
	\int_{\Omega} (\nabla u \cdot \nabla\xi + V^{+} u\xi){}
	= \int_{\Omega} \widehat{\xi} \dif\nu{}
	\quad \text{for every \(\xi \in W_{0}^{1, 2}(\Omega) \cap L^{\infty}(\Omega)\).}
	\]	
	Given a sequence of radial mollifiers \((\rho_{n})_{n \in \N}\) in \(C_{c}^{\infty}(B_{1})\), let \(u_{n}\) be the duality solution of \eqref{eqDualitySolutionV+} with datum \(\rho_{n} * \nu\), so that by uniqueness \(u_{n} = \zeta_{\rho_{n} * \nu}\)\,.{}
	By \eqref{eq789}, for every \(n \in \N\) we then have 
	\begin{equation}
		\label{eq895}
		\widehat{u_{n}} = 0
		\quad \text{in \(S\).}
	\end{equation}
	
	We claim that 
	\begin{equation}
		\label{eq901}
		u_{n} \to u
		\quad \text{in \(W_{0}^{1, 2}(\Omega)\).}
	\end{equation}
	To this end, we subtract the Euler-Lagrange equations satisfied by \(u_{n}\) and \(u\), and apply \(u_{n} - u\) as test function. 
	Using the nonnegativity of \(V^{+}\), we get
	\begin{equation}
	\label{eq3211}
	\int_{\Omega}{\abs{\nabla(u_{n} - u)}^{2}}
	\le \int_{\Omega}{(u_{n} - u) \rho_{n} * \nu} - \int_{\Omega}{\widehat{u_{n} - u} \dif\nu}.
	\end{equation}
	Since the sequence \((u_{n})_{n \in \N}\) is bounded in \(W_{0}^{1, 2}(\Omega)\) and converges to \(u\) in \(L^{1}(\Omega)\), we have weak convergence in \(W_{0}^{1, 2}(\Omega)\).{} 
	From \eqref{eq822} with \(U = \Omega\), we then get
	\[{}
	\lim_{n \to \infty}{\int_{\Omega}{\widehat{u_{n} - u} \dif\nu}}
	= 0.
	\]
	For the first integral in the right-hand side of \eqref{eq3211}, we apply Fubini's theorem,
	\begin{equation}
		\label{eq909}
	\int_{\Omega}{(u_{n} - u) \rho_{n} * \nu}
	= \int_{\Omega}{{\rho_{n}} * (u_{n} - u)  \dif\nu}.
	\end{equation}
	Extending \(u_{n} - u\) by \(0\) in \(\R^{N} \setminus \Omega\)\,, we have that \(u_{n} - u \in W^{1, 2}(\R^{N})\).
	The sequence \(({\rho_{n}} * (u_{n} - u))_{n \in \N}\)  is supported in a smooth bounded open set \(U \Supset \Omega\) and converges weakly to zero in \(W^{1, 2}_{0}(U)\).{}
	From \eqref{eq822} and \eqref{eq909}, we get
	\[{}
	\lim_{n \to \infty}{\int_{\Omega}{(u_{n} - u) \rho_{n} * \nu}}
	= 0.
	\]
	Thus, as \(n \to \infty\) in \eqref{eq3211},
	\[{}
	\lim_{n \to \infty}{\int_{\Omega}{\abs{\nabla(u_{n} - u)}^{2}}} 
	= 0,
	\]
	which implies \eqref{eq901}.
	Now passing to a subsequence \((u_{n_{j}})_{j \in \N}\)\,, one deduces that 
	\[{}
	\widehat{u_{n_{j}}} \to \widehat{u}
	\quad \text{quasi-everywhere in \(\Omega\).}
	\]
	In view of \eqref{eq895}, the conclusion holds when \(F * |\nu|\) is bounded.
	
	In the case of a general measure \(\nu \in \mathcal{M}(\Omega)\), it suffices to prove that the truncated function \(T_{1}(u)\) satisfies the conclusion.
	Observe that \(T_{1}(u) \in W_{0}^{1, 2}(\Omega) \cap L^{2}(\Omega; V^{+} \dif x)\) and
	\[{}
	- \Delta T_{1}(u) + V^{+} T_{1}(u)
	 = \widetilde{\nu}
	 \quad \text{in the sense of distributions in \(\Omega\),}
	\]
	for some diffuse measure \(\widetilde\nu \in \mathcal{M}(\Omega)\)\,; see \cites{DMOP,Brezis_Ponce:2004,Brezis_Ponce:2008}.{}
	By a classical property in Potential Theory~\cite{Helms:2009}*{Theorem~3.6.3} applied to \(\widetilde\nu^{+}\) and \(\widetilde\nu^{-}\), this measure \(\widetilde\nu\) can be strongly approximated in \(\mathcal{M}(\Omega)\) by a sequence of measures \((\nu_{j})_{j \in \N}\) such that \(|\nu_{j}|\) has bounded Newtonian potential, for which the conclusion holds from the first part of the proof.
	
	Denote by \(v_{n}\) the duality solution of \eqref{eqDualitySolutionV+} with datum \(\nu_{n}\).{}
	By uniqueness of duality solutions, \(v_{n}\) and \(T_{1}(u)\) are also minimizers of their associated energy functionals and therefore satisfy an Euler-Lagrange equation.
	Subtracting those equations, we then have
	\[{}
	\int_{\Omega} \nabla(v_{n} - T_{1}(u)) \cdot \nabla\xi + V^{+}(v_{n} - T_{1}(u))\xi{}
	= \int_{\Omega} \widehat\xi \dif(\nu_{n} - \widetilde\nu),
	\]
	for every \(\xi \in W_{0}^{1, 2}(\Omega) \cap L^{2}(\Omega; V^{+} \dif x)\).
	We now fix \(k > 0\).{}
	Applying the identity above with test function \(\xi = T_{k}(v_{n} - T_{1}(u))\), by nonnegativity of \(V^{+}\) we get
	\[{}
	\int_{\Omega} \bigl|\nabla\bigl(T_{k}(v_{n} - T_{1}(u))\bigr)\bigr|^{2} 
	\le \int_{\Omega} \widehat{T_{k}(v_{n} - T_{1}(u))} \dif(\nu_{n} - \widetilde\nu)
	\le k \int_{\Omega} \dif|\nu_{n} - \widetilde\nu|.
	\]
	Thus, as \(n \to \infty\),
	\[{}
	T_{k}(v_{n} - T_{1}(u)) \to 0
	\quad \text{in \(W_{0}^{1, 2}(\Omega)\).}
	\]
	Since \(\widehat{v_{n}} = 0\) quasi-everywhere in \(S\), we conclude that 
	\[{}
	T_{1}(\widehat{u})
	= \widehat{T_{1}(u)}
	= 0 
	\quad \text{quasi-everywhere in \(S\).}
	\qedhere
	\]
\end{proof}


\section{Localization on $D_{i}$}
\label{sectionCutoff}

In this section, we prove that duality solutions of \eqref{eqDualitySolutionV+} and functions in $W^{1,2}_{0}(\Omega) \cap L^{2}(\Omega; V^{+} \dif x)$ can be localized on each component $D_{i}$\,.

\begin{proposition}\label{cutoffV+}
Let \(i \in I\).{}
If $u$ is a duality solution of \eqref{eqDualitySolutionV+} with datum \(\nu \in \mathcal{M}(\Omega)\), then
 $u\chi_{D_i}$ is a duality solution of \eqref{eqDualitySolutionV+} with datum $\nu\lfloor_{D_i}{}$.
\end{proposition}

We rely on an orthogonality property of the Green's functions \(G_{x}\) that is implicitly used in \cite{OP} to identify the components \(D_{i}\)\,.{}
It relies on \eqref{eq826} above and \cite{OP}*{Proposition~9.1}, and is valid for each \(i \in I\):{}
\begin{enumerate}[\((i)\)]
	\item 
	\label{item-1012}
	For every \(x \in \Omega \setminus D_{i}\)\,, we have \(\widehat{G_{x}} = 0\) in \(D_{i}\)\,;{}
	\item 
	\label{item-1015}
	For every \(x \in D_{i}\)\,, we have \(\widehat{G_{x}} = 0\) in \(\Omega \setminus D_{i}\)\,.
\end{enumerate}

\begin{proof}[Proof of Proposition~\ref{cutoffV+}]
By the representation formula \eqref{eqRepresentationFormula}, the duality solution \(v\) of
\begin{equation}
	\label{eq897}
	\left\{
	\begin{alignedat}{2}
	-\Delta v + V^{+} v & = \nu\lfloor_{D_{i}} && \quad \text{in } \Omega,\\
	 v & = 0 && \quad  \text{on }  \partial \Omega,
	\end{alignedat}
	\right.
\end{equation}
satisfies
\begin{equation}
	\label{eq907}
v(x)
= \int_{\Omega}\widehat{G_x} \dif\nu\lfloor_{D_{i}}{}
= \int_{D_{i}}\widehat{G_x} \dif\nu
\quad
\text{for almost every \(x \in \Omega\)}.
\end{equation}
By \((\ref{item-1012})\), 
\begin{equation}
	\label{eq917}
\int_{D_{i}}\widehat{G_x} \dif\nu = 0
\quad \text{for every \(x \in \Omega \setminus D_{i}\)}
\end{equation}
and, by \((\ref{item-1015})\),
\begin{equation}
	\label{eq921}
\int_{D_{i}}\widehat{G_x} \dif\nu{}
= \int_{\Omega}\widehat{G_x} \dif\nu
\quad \text{for every \(x \in D_{i}\)\,.}
\end{equation}
We conclude from \eqref{eqRepresentationFormula} and \eqref{eq907}, \eqref{eq917} and \eqref{eq921} that \(v = u \chi_{D_{i}}\) almost everywhere in \(\Omega\).{}
Hence, \(u\chi_{D_{i}}\) is the duality solution of \eqref{eq897}.
\end{proof}

As a consequence of Proposition~\ref{cutoffV+}, we deduce a localized strong maximum principle for duality solutions of \eqref{eqDualitySolutionV+}.
We first recall that the strong maximum principle holds for $\zeta_{g}$ with nonnegative $g \in L^{\infty}(\Omega)$ in the sense that, for each $i \in I$, we have that either $\widehat{\zeta_{g}} > 0$ on $D_{i}$ or $\widehat{\zeta_{g}} \equiv 0$ on $D_{i}$\,; see \cite{OP}*{Theorem~12.1}.
Next, every duality solution $u$ of \eqref{eqDualitySolutionV+} with nonnegative data $\nu$ can be bounded from below by some duality solution with nonnegative datum in \(L^{\infty}(\Omega)\). 
More precisely, by \cite{OP}*{Proposition~7.1}, there exists a bounded nondecreasing continuous function \(Q : \R_{+} \to \R_{+}\) with \(Q(t) > 0\) for \(t > 0\) such that 
\begin{equation}
	\label{eq1399}
u \geq \zeta_{Q(u)}
\quad \text{almost everywhere in \(\Omega\).}
\end{equation}
A valid choice of function \(Q\) satisfying \eqref{eq1399} is $Q(t) = \frac{\alpha-1}{C\alpha}\min{\{t^\alpha,1\}}$, where $\alpha>1$ and $C > 0$ is a constant that depends on \(\alpha\) and \(\Omega\).

\begin{corollary}\label{MaximumPrinciple}
Let \(i \in I\) and let $u$ be a nonnegative duality solution of \eqref{eqDualitySolutionV+} with datum \(\nu \in \mathcal{M}(\Omega)\). 
If $\nu \geq 0$ in $D_{i}$ and \(\int_{D_{i}} u > 0\), then
\[{}
\liminf_{r \to 0}{\fint_{B_{r}(x)} u} > 0
\quad \text{for every \(x \in D_{i}\)\,.}
\]
\end{corollary}

\begin{proof}
Let $\bar{u}_{i} = u \chi_{D_{i}}$\,. 
Then, $\bar{u}_{i}$ is a duality solution of \eqref{eqDualitySolutionV+} with nonnegative datum $\nu\lfloor_{D_{i}}$\,. By the comparison principle \eqref{eq1399}, 
\[
u \ge \bar{u}_{i} \ge \zeta_{Q(\bar{u}_{i})} \quad \text{almost everywhere in $\Omega$.}
\]
Since \(\int_{D_{i}} u > 0\), the function $Q(\bar{u}_{i})$ is nonzero on $D_{i}$\,. 
Then, by the strong maximum principle, we have $\widehat{\zeta_{Q(\bar{u}_{i})}} > 0$ on $D_{i}$\,. 
Thus, for every $x \in D_{i}$\,,
\[
\liminf_{r \to 0}{\fint_{B_{r}(x)} u}
\ge
\lim_{r \to 0}{\fint_{B_{r}(x)} \zeta_{Q(\bar{u}_{i})}}
=
\widehat{\zeta_{Q(\bar{u}_{i})}}(x)
>
0. \qedhere
\]
\end{proof}

We now prove that decomposition~\eqref{eqDecomposition} of \(\Omega \setminus S\) in terms of its components \(D_{i}\) induces a natural splitting \eqref{eqDecompositionFunction} of functions in \(W_{0}^{1, 2}(\Omega) \cap L^{2}(\Omega; V^{+} \dif x)\).{}

\begin{proposition}
	\label{propositionSobolevCutoff}
	If \(\xi \in W_{0}^{1, 2}(\Omega) \cap L^{2}(\Omega; V^{+} \dif x)\), then, for every \(i \in I\), we have \(\xi\chi_{D_{i}} \in W_{0}^{1, 2}(\Omega)\) and
\[{}
\xi = \sum_{i \in I}{\xi \chi_{D_{i}}}
\quad \text{almost everywhere in \(\Omega\).}
\]	
\end{proposition}

We begin with the following

\begin{lemma} \label{decomposition}
	For every \(\xi \in  W_{0}^{1, 2}(\Omega) \cap L^{2}(\Omega; V^{+} \dif x)\), we have \(\xi = 0\) almost everywhere in \(S\).
\end{lemma}

\begin{proof}[Proof of Lemma~\ref{decomposition}]
	We first observe that
	\begin{equation}
		\label{eq788}
		\zeta_{\chi_{S}} = 0
		\quad \text{almost everywhere in \(\Omega\).}
	\end{equation}
	Indeed, since \(\zeta_{1} = 0\) almost everywhere in \(S\), by \eqref{eq675} we have
	\[{}
	\int_{\Omega} \zeta_{\chi_{S}}
	= \int_{\Omega} \zeta_{1}\, \chi_{S}
	= 0.
	\]
	Since \(\zeta_{\chi_{S}}\) is nonnegative, \eqref{eq788} follows.

	Given \(\xi \in W_{0}^{1, 2}(\Omega) \cap L^{2}(\Omega; V^{+} \dif x)\), let us assume by contradiction that \(S \cap \{\xi \ne 0\}\) has positive Lebesgue measure. 
	Hence, \(\int_{\Omega} \abs{\xi} \, \chi_{S} > 0\).
	Computing the functional
	\[{}
	\eta \in  W_{0}^{1, 2}(\Omega) \cap L^{2}(\Omega; V^{+} \dif x)
	\longmapsto \frac{1}{2} \int_{\Omega} (|\nabla \eta|^{2} + V^{+}\eta^{2}) - \int_{\Omega} \eta \, \chi_{S}
	\]
	on \(t \abs{\xi}\) for \(t > 0\) small, one concludes that its infimum is negative. 
	Thus, the minimizer \(\zeta_{\chi_{S}}\) is nontrivial, in contradiction with \eqref{eq788}.
	Therefore, the set \(S \cap \{\xi \ne 0\}\) has Lebesgue measure zero.
\end{proof}

Lemma~\ref{decomposition} suffices for our purposes and is trivially satisfied when \(S\) is negligible for the Lebesgue measure.
A small adaptation of the proof yields a sharper version in terms of the \(W^{1, 2}\)~capacity:

\begin{proposition}
	\label{propDecompositionCapacity}
	For every \(\xi \in  W_{0}^{1, 2}(\Omega) \cap L^{2}(\Omega; V^{+} \dif x)\), we have \(\widehat{\xi} = 0\) quasi-everywhere in \(S\).
\end{proposition}

\begin{proof}[Proof of Proposition~\ref{propDecompositionCapacity}]
	We assume by contradiction that \(S \cap \{\widehat{\xi} \ne 0\}\) has positive \(W^{1, 2}\)~capacity. 
	Then, there exist a compact set \(K \subset S \cap \{\widehat{\xi} \ne 0\}\) with positive capacity and a nonnegative diffuse measure \(\nu \in \mathcal{M}(\Omega) \cap (W_{0}^{1, 2}(\Omega))'\) supported in \(K\) with \(\nu(K) > 0\)\,; see \cite{bookponce}*{Proposition~A.17}.{}
	We then have that \(\int_{\Omega} \widehat{\xi} \dif\nu > 0\) and the functional
	\[{}
	\eta \in  W_{0}^{1, 2}(\Omega) \cap L^{2}(\Omega; V^{+} \dif x)
	\longmapsto \frac{1}{2} \int_{\Omega} (|\nabla \eta|^{2} + V^{+}\eta^{2}) - \int_{\Omega} \widehat{\eta} \, \dif\nu
	\]
	has a nontrivial nonnegative minimizer \(v\), which is the duality solution of \eqref{eqDualitySolutionV+}.
	In particular, since \(\nu\) is supported in \(S\),
	\[{}
	\int_{\Omega} v
	= \int_{\Omega} \widehat{\zeta_{1}} \dif\nu
	= 0.
	\]
	We thus have a contradiction, whence \(S \cap \{\widehat{\xi} \ne 0\}\) has \(W^{1, 2}\)~capacity zero. 
\end{proof}

\begin{proof}[Proof of Proposition~\ref{propositionSobolevCutoff}]
	Since 
	\[{}
	\xi = \sum_{i \in I}{\xi \chi_{D_{i}}} + \xi \chi_{S}
	\]
	and, by Lemma~\ref{decomposition}, \(\xi = 0\) almost everywhere in \(S\), we are left to prove that \(\xi \chi_{D_{i}} \in W_{0}^{1, 2}(\Omega)\) for every \(i \in I\).

	Given \(\Phi \in C^{\infty}(\overline{\Omega}; \R^{N})\), let \(v\) be the duality solution of \eqref{eqDualitySolutionV+} with datum \(\nu = \Div{\Phi} \dif x\).
	Then, by Proposition~\ref{cutoffV+}, \(\bar v_{i} \vcentcolon= v \chi_{D_{i}}\) is the duality solution of \eqref{eqDualitySolutionV+} with datum \(\nu = (\Div{\Phi})\chi_{D_{i}} \dif x\).
	By uniqueness, \(v\) and \(\bar v_{i}\) are also variational solutions for the same data. 
	In particular, they both belong to \(W_{0}^{1, 2}(\Omega) \cap L^{2}(\Omega; V^{+} \dif x)\) and satisfy the associated Euler-Lagrange equations.

	We now observe that
	\begin{equation}
		\label{eq1049}
	\nabla \bar v_{i} = (\nabla v) \chi_{D_{i}}
	\quad \text{almost everywhere in \(\Omega\).}
	\end{equation}
	Indeed, this follows from the facts that \(\nabla \bar v_{i} = 0\) almost everywhere in \(\{\bar v_{i} = 0\} \supset \Omega \setminus D_{i}\) and
	\(\nabla (\bar v_{i} - v) = 0\) almost everywhere in \(\{\bar v_{i} - v = 0\} \supset D_{i}\).{}
	For \(\xi \in W_{0}^{1, 2}(\Omega) \cap L^{2}(\Omega; V^{+} \dif x)\), using the Euler-Lagrange equation satisfied by \(\bar v_{i}\) and \eqref{eq1049} we then have
	\[{}
	\int_{\Omega} \xi\chi_{D_{i}} \Div{\Phi}
	= \int_{\Omega} \nabla \bar v_{i} \cdot \nabla \xi + V^{+} \bar v_{i} \xi{}
	= \int_{D_{i}} \nabla v \cdot \nabla \xi + V^{+} v \xi.
	\]  
	Therefore,
	\begin{equation}
		\label{eq1147}
	\biggl| \int_{\Omega} \xi\chi_{D_{i}} \Div{\Phi} \biggr|{}
	\le \|v\| \|\xi\|,
	\end{equation}
	where we use the norm
	\[{}
	\|\eta\| \vcentcolon= \biggl[ \int_{\Omega} (\abs{\nabla\eta}^{2} + V^{+} \eta^{2}) \biggr]^{\frac{1}{2}}.
	\]
	Applying \(v\) as a test function in the Euler-Lagrange equation satisfied by \(v\) itself, we have 
	\[{}
	\|v\|^{2}
	 = \int_{\Omega} v \Div{\Phi}
	 = - \int_{\Omega} \nabla v \cdot \Phi{}
	 \le \|v\|\|\Phi\|_{L^{2}(\Omega)}.
	\]
	Thus, \(\|v\| \le \|\Phi\|_{L^{2}(\Omega)}\).{}
	Inserting this estimate in \eqref{eq1147} we get
	\begin{equation}
		\label{eq1165}
	\biggl| \int_{\Omega} \xi\chi_{D_{i}} \Div{\Phi} \biggr|{}
	\le \|\Phi\|_{L^{2}(\Omega)} \|\xi\|{}
	\quad \text{for every \(\Phi \in C^{\infty}(\overline{\Omega}; \R^{N})\).}
	\end{equation}
	We deduce from the Riesz Representation Theorem that \(\nabla(\xi\chi_{D_{i}}) \in L^{2}(\Omega; \R^{N})\).{}
	Since \(\Omega\) is smooth and \eqref{eq1165} holds for every smooth \(\Phi\) that need not have compact support in \(\Omega\), we conclude that \(\xi\chi_{D_{i}} \in W_{0}^{1, 2}(\Omega)\).
\end{proof}


\section{Duality solutions for signed potentials}
\label{sectionDuality}

We begin by extending the definition of duality solutions to the Dirichlet problem
\begin{equation}
	\label{eqDualitySolution}
	\left\{
	\begin{alignedat}{2}
	-\Delta  u+ V u & = \mu && \quad \text{in } \Omega,\\
	 u & = 0 && \quad  \text{on }  \partial \Omega,
	\end{alignedat}
	\right.
\end{equation}
with a signed potential \(V\):

\begin{definition}
Given $\mu\in \mathcal{M}(\Omega)$, we say that $u \in L^{1}(\Omega)$ is a \emph{duality solution} of \eqref{eqDualitySolution}
whenever \(V^{-}u \in L^{1}(\Omega)\) and
\[
\int_{\Omega}u f 
= \int_{\Omega} \widehat{\zeta_f} \dif\mu + \int_{\Omega} \zeta_f V^{-}u  \quad \text{for every } f \in L^{\infty}(\Omega),
\]
where \(\zeta_{f}\) is the minimizer of \eqref{eq731}.
\end{definition}

It is convenient to observe that, as the test functions \(\zeta_{f}\) depend on \(V^{+}\) and \(\widehat{\zeta_{f}} = \zeta_{f}\) almost everywhere in \(\Omega\), a duality solution of \eqref{eqDualitySolution} is also a duality solution of 
\begin{equation}
	\label{eqDualitySolutionPositivePotential}
	\left\{
	\begin{alignedat}{2}
	-\Delta  u+ V^{+} u & = \mu + V^{-} u && \quad \text{in } \Omega,\\
	 u & = 0 && \quad  \text{on }  \partial \Omega,
	\end{alignedat}
	\right.
\end{equation}
where \(\mu + V^{-} u\) is regarded as the datum of the problem.
In particular, the precise representative \(\widehat{u}\) is also well-defined quasi-everywhere in \(\Omega\) and we have the following direct consequences of Propositions~\ref{propositionSZero} and~\ref{cutoffV+}, respectively:

\begin{corollary}
	\label{propositionSZeroCorollary}
	If \(u\) is a duality solution of \eqref{eqDualitySolution} with datum \(\mu \in \mathcal{M}(\Omega)\), then
	\(\widehat{u} = 0\) quasi-everywhere in \(S\).
\end{corollary}

\begin{corollary}
\label{cutoff}
Let \(i \in I\).{}
If $u$ is a duality solution of \eqref{eqDualitySolution} with datum \(\mu \in \mathcal{M}(\Omega)\), then
 $u\chi_{D_i}$ is a duality solution of \eqref{eqDualitySolution} with datum $\mu\lfloor_{D_i}{}$\,.
\end{corollary}

We now rely upon \cite{OP}*{Proposition~4.1} for the operator \(-\Delta + V^{+}\) to deduce that a duality solution is a distributional solution with possibly different datum:

\begin{proposition}
	\label{propositionDictionary}
	If \(u\) is a duality solution of \eqref{eqDualitySolution} with datum \(\mu \in \mathcal{M}(\Omega)\), then \(u \in W_{0}^{1, 1}(\Omega)\) and
	\[{}
	- \Delta u + Vu = \mu\lfloor_{\Omega \setminus S}{} - \tau{}
	\quad \text{in the sense of distributions in \(\Omega\),}
	\] 
	where \(\tau \in \mathcal{M}(\Omega)\) is a diffuse measure such that \(|\tau|(\Omega \setminus S) = 0\).
\end{proposition}

Since \(\tau\) is diffuse and carried by \(S\), under the assumption 
\begin{equation}
	\label{eqIntroductionCapacity}
	\capt_{W^{1, 2}}{(S)} = 0,
\end{equation}  
it then follows that \(\tau = 0\).
The smallness condition \eqref{eqIntroductionCapacity} is directly verified with the existence of \(\xi \in W_{0}^{1, 2}(\Omega) \cap L^{2}(\Omega; V^{+} \dif x)\) whose level set \(\{\widehat{\xi} = 0\}\) has \(W^{1, 2}\) capacity zero; see Proposition~\ref{propDecompositionCapacity}.

\begin{proof}[Proof of Proposition~\ref{propositionDictionary}]
	Since \(u\) is a duality solution of \eqref{eqDualitySolutionPositivePotential}
	involving a finite measure in \(\Omega\), we have \(u \in W_{0}^{1, 1}(\Omega)\).{}
	Take the duality solutions \(u_{1}\) and \(u_{2}\) of \eqref{eqDualitySolutionV+} with data \((\mu + V^{-}u \dif x)^{+}\) and \((\mu + V^{-}u \dif x)^{-}\), respectively.
	By \cite{OP}*{Proposition~4.1}, there exist nonnegative diffuse measures \(\tau_{1}\) and \(\tau_{2}\) such that \(\tau_{1}(\Omega \setminus S) = \tau_{2}(\Omega \setminus S) = 0\) with
	\[{}
	- \Delta u_{1} + V^{+}u_{1} = (\mu + V^{-}u \dif x)^{+}\lfloor_{\Omega \setminus S}{} - \tau_{1}{}
	\] 
	and	
	\[{}
	- \Delta u_{2} + V^{+}u_{2} = (\mu + V^{-}u \dif x)^{-}\lfloor_{\Omega \setminus S}{} - \tau_{2}{}
	\] 
	in the sense of distributions in \(\Omega\).{}
	From Proposition~\ref{propositionSZeroae}, one has \(u = 0\) almost everywhere in \(S\). 
	Thus,
	\begin{equation}
	\label{eq-1001}
	V^{-}u\chi_{\Omega \setminus S}
	= V^{-}u
	\quad \text{almost everywhere in \(\Omega\).}
	\end{equation}
	Subtracting the equations satisfied by \(u_{1}\) and \(u_{2}\), and using \eqref{eq-1001}, we then get
	\begin{equation*}
	- \Delta u + V^{+}u = \mu\lfloor_{\Omega \setminus S}{} + V^{-}u - \tau{}
	\quad \text{in the sense of distributions in \(\Omega\),}
	\end{equation*}
	where \(\tau \vcentcolon= \tau_{1} - \tau_{2}\).
\end{proof}

Even for a nonnegative potential $V$ the measure $\tau$ which appears in Proposition~\ref{propositionDictionary} can be nonzero:

\begin{example}
	\label{exampleOrsinaPonce}
	Given \(\alpha \ge 1\), let \(V : B_{1} \to [0, +\infty]\) be defined for \(x = (x_{1}, \ldots, x_{N})\) by
	\[{}
	V(x) = \frac{1}{|x_{1}|^{\alpha}}.
	\]
	In this case, \(S = B_{1} \cap  \{x_{1} = 0\}\) and then \(B_{1} \setminus S\) has two connected components.  
	Since \(\zeta_{1}\) is a duality solution with constant datum \(1\), we have
	\[
	- \Delta \zeta_{1} + V \zeta_{1} = 1 - \tau
	\quad \text{in the sense of distributions in \(B_{1}\)\,,}
	\]
	for some finite measure \(\tau\) supported in \(\{x_{1} = 0\}\).{}
	When \(\alpha \ge 2\), we have \(\tau = 0\), but when \(1 \le \alpha < 2\) the Hopf lemma holds in each component of \(B_{1} \setminus S\) and implies that \(\tau(S) > 0\)\,; see \cite{OrsinaPonce:2008}*{Theorem~9.1}.
\end{example}

\begin{corollary}
	If \(u\) is a nonnegative distributional solution of \eqref{eqDualitySolution} with datum \(\mu \in \mathcal{M}(\Omega)\), then \(\mu\lfloor_{S}{} \le 0\).
\end{corollary}

\begin{proof}
	We recall that a distributional solution is also a duality solution for the same datum.
	By comparison with the conclusion of Proposition~\ref{propositionDictionary}, we deduce that \(\mu\lfloor_{S}{} = -\tau\) in the sense of distributions in \(\Omega\), and then equality holds as measures; see \cite{bookponce}*{Proposition~6.12}. 
	We are left to show that \(\tau \ge 0\).{}
	To this end, we first identify the diffuse part \((\Delta u)_{\mathrm{d}}\) of the measure \(\Delta u\) with respect to the \(W^{1, 2}\) capacity.
	Since \(Vu = 0\) almost everywhere in \(S\) and \(\tau = -\mu\lfloor_{S}{}\) is diffuse, we have
	\begin{equation}
		\label{eq-1036}
	(\Delta u)_{\mathrm{d}}
	= \tau{}
	\quad \text{in \(S\).}
	\end{equation}
	On the other hand, as a consequence of Kato's inequality, by nonnegativity of \(u\) in \(\Omega\) we have
	\begin{equation}
		\label{eq-1043}
	(\Delta u)_{\mathrm{d}}
	\ge 0
	\quad \text{in \(\{\widehat{u} = 0\}\)\,;}
	\end{equation}
	see \cite{bookponce}*{Eq.~(6.5)}.
	By Corollary~\ref{propositionSZeroCorollary}, we have \(\widehat{u} = 0\) quasi-everywhere in \(S\).{}
	It thus follows from \eqref{eq-1036} and \eqref{eq-1043} that \(\mu\lfloor_{S}{} = - \tau \le 0\).
\end{proof}

\section{Approximation scheme}\label{Approximation}

Existence of a unique duality solution of \eqref{eqDualitySolutionV+} entitles us to construct a sequence of duality solutions for a family of truncated problems associated to \eqref{eqDualitySolution} with signed \(V\). 

\begin{proposition}\label{increasingsequence}
Let \(\mu\) be a nonnegative measurable function on $\Omega$ and let \((u_{n})_{n \in \N}\) be the sequence defined by induction as \(u_{0} = 0\) and, for \(n \ge 1\), \(u_{n}\) is the duality solution of
\begin{equation*}
\left\{
\begin{alignedat}{2}
-\Delta  u_{n} + V^+ u_{n}
& = T_{n}(\mu) + T_{n}(V^-) u_{n-1} && \quad \text{in } \Omega,\\
u_{n} 
& = 0 && \quad \text{on }  \partial\Omega.
\end{alignedat}
\right.
\end{equation*}
Then, for every \(n \in \N\), we have \(u_{n} \in W_{0}^{1, 2}(\Omega) \cap L^{2}(\Omega; V^{+} \dif x) \cap L^{\infty}(\Omega)\) and
\[
0\le u_{n}\le u_{n+1} \quad \text{almost everywhere in $\Omega$}.
\]
\end{proposition}

\begin{proof}
For every nonnegative \(f \in L^{\infty}(\Omega)\), we have \(\zeta_{f} \ge 0\) almost everywhere in \(\Omega\).
Thus, by nonnegativity of \(\mu\),
\[{}
\int_{\Omega} u_{1}f
= \int_{\Omega} T_{1}(\mu) \zeta_{f}
\ge 0.
\]
Hence, \(u_{1} \ge 0 = u_{0}\) almost everywhere in \(\Omega\).
We next assume that the conclusion holds for some \(n \in \N\).{}
By nonnegativity of \(\mu\), \(T_{n+2}(\mu) \ge T_{n+1}(\mu)\).
Then, for every nonnegative \(f \in L^{\infty}(\Omega)\),
\[{}
\int_{\Omega} (u_{n + 2} - u_{n + 1}) f
\ge \int_{\Omega} (T_{n + 2}(V^{-})u_{n + 1} - T_{n + 1}(V^{-})u_{n}) \zeta_{f}.
\]
By the induction assumption, the integrand in the right-hand side is nonnegative and we deduce that $u_{n+2} - u_{n+1} \ge 0$ almost everywhere in \(\Omega\).{}
Finally, since $T_{n}(\mu) \in L^{\infty}(\Omega)$, we have by induction that $u_{n} \in L^{\infty}(\Omega)$ for every \(n \in \N\).{}
As \(u_{n} \in L^1(\Omega; V^+\dif x)\), we then have \(u_{n} \in L^2(\Omega; V^+\dif x)\).
Finally, since \(\Delta u_{n}\) is a finite measure in \(\Omega\), we also have by interpolation that $u_{n} \in W_{0}^{1,2}(\Omega)$.
\end{proof}


\section{Variational setting}

Throughout the section, we assume that there exists a measurable function  \(w_{i} : D_{i} \to (0, +\infty)\) such that \eqref{eqPoincareStrong} holds, whence
\[{}
\mathcal{H}_{i}(\Omega) \subset L^{2}(D_{i}; w_{i} \dif x).
\]
Denoting by \(E\) the functional defined in \(\mathcal{H}_{i}(\Omega)\) by  \eqref{eqFunctional}, the following property is standard in the Calculus of Variations:

\begin{proposition}\label{existence-sequence} 
For every $h\in L^2(D_{i}; w_{i} \dif x)$, the functional \(E\) has a unique minimizer \(\theta_{i, h}\) and any minimizing sequence of \(E\) is a Cauchy sequence in \(\mathcal{H}_{i}(\Omega)\).
\end{proposition}
\begin{proof}
Uniqueness of the minimizer follows from the parallelogram identity satisfied by the norm \(\|\cdot\|_{i}\) which yields
\begin{equation}
	\label{eqParallelogram}
\frac{1}{2} \|u - v\|_{i}^{2}
= E(u) + E(v) - 2 E\Bigl( \frac{u + v}{2} \Bigr){}
\quad \text{for every \(u, v \in \mathcal{H}_{i}(\Omega)\).}
\end{equation}
If \(u\) and \(v\) are both minimizers of \(E\), then \(\|u - v\|_{i}^{2} \le 0\).

We now observe that \(E\) is bounded from below, whence \(\inf{E} \in \R\).
Taking a minimizing sequence  $(\xi_j)_{j\in \mathbb{N}}$\,, it follows from \eqref{eqParallelogram} that, for every \(j, k \in \N\),
\[{}
\frac{1}{2} \|\xi_{j} - \xi_{k}\|_{i}^{2}
= E(\xi_{j}) + E(\xi_{k}) - 2 E\Bigl( \frac{\xi_{j} + \xi_{k}}{2} \Bigr){}
\le E(\xi_{j}) + E(\xi_{k}) - 2 \inf{E}.
\]
Observe that the right-hand side converges to zero as \(j, k \to \infty\).{}
Hence, $(\xi_j)_{j\in \mathbb{N}}$ is a Cauchy sequence and converges in \(\mathcal{H}_{i}(\Omega)\) to the unique minimizer of \(E\).
\end{proof}

\begin{proposition}
	\label{propPositive}
	If $h\in L^2(D_{i}; w_{i} \dif x)$ is nonnegative, then \(\theta_{i, h} \ge 0\) almost everywhere in \(D_{i}\)\,.	
\end{proposition}

\begin{proof}
Let \((\xi_{j})_{j \in \N}\) be a minimizing sequence in \(H_{i}(\Omega)\) of the functional \(E\).{}
For each \(j \in \N\), we write
\[{}
E(\xi_{j})
= E(\xi_{j}^{+}) + E(-\xi_{j}^{-}).
\]
Since \(h\) is nonnegative and \eqref{eqPoincareStrong} holds, we have \(E(-\xi_{j}^{-}) \ge 0\) and then
\[{}
E(\xi_{j}) \ge E(\xi_{j}^{+})
\]
Thus, \((\xi_{j}^{+})_{j \in \N}\) is also a minimizing sequence.
By Proposition~\ref{existence-sequence}, it converges to \(\theta_{i, h}\) in \(\mathcal{H}_{i}(\Omega)\) and then also in \(L^{2}(D_{i}; w_{i} \dif x)\).{}
In particular, \(\theta_{i, h} \ge 0\) almost everywhere in \(D_{i}\)\,.
\end{proof}

As an alternative to the previous proof based on minimizing sequences, one may rely on the decomposition \(\theta_{i, h} = \theta_{i, h}^{+} - \theta_{i, h}^{-}\) in \(\mathcal{H}_{i}(\Omega)\) as a difference between positive and negative parts.
It then suffices to use \(\theta_{i, h}^{-}\) in the Euler-Lagrange equation to deduce in a standard way that \(\theta_{i, h}^{-} = 0\).{}
The possibility of such an approach is ensured by the following

\begin{proposition}
	\label{propositionMinMax}
	Let \(i \in I\).{}
	If \(u, v \in \mathcal{H}_{i}(\Omega)\), then there exist \(f, g \in \mathcal{H}_{i}(\Omega)\) such that
	\(f = \min{\{u, v\}}\) and
	\(g = \max{\{u, v\}}\)
	almost everywhere in \(D_{i}\)
	and
	\[{}
	\norm{f}_{i}^{2} + \norm{g}_{i}^{2}
	\le \norm{u}_{i}^{2} + \norm{v}_{i}^{2}.
	\]
\end{proposition}

	We rely on the following closure property which is a convenient tool to identify the weak limit of a sequence in \(\mathcal{H}_{i}(\Omega)\)\,:

\begin{lemma}
	\label{lemmaClosureH}
	Let \((v_{n})_{n \in \N}\) be a sequence in \(\mathcal{H}_{i}(\Omega)\). 
	If \(v_{n} \to v\) almost everywhere in \(D_{i}\) and \(v_{n} \rightharpoonup \widetilde{v}\) in \(\mathcal{H}_{i}(\Omega)\), then \(v = \widetilde{v}\) almost everywhere in \(D_{i}\)\,.
\end{lemma}

\begin{proof}[Proof of Lemma~\ref{lemmaClosureH}]
	By the Banach--Saks property in a Hilbert space (which is a more comfortable and precise version of Mazur's lemma), we can take a subsequence \((v_{n_{j}})_{j \in \N}\) that converges strongly to \(\widetilde{v}\) in \(\mathcal{H}_{i}(\Omega)\) in the Cesàro sense, that is
	\[{}
	\frac{1}{N + 1}\sum_{j = 0}^{N}{v_{n_{j}}} \to \widetilde{v}
	\quad \text{in \(\mathcal{H}_{i}(\Omega)\),}
	\]
	and then also in $L^{2}(D_{i}; w_{i} \dif x)$.
	By pointwise convergence of \((v_{n_{j}})_{j \in \N}\) to \(v\) in \(D_{i}\)\,, we deduce that \(v = \widetilde{v}\) almost everywhere in \(D_{i}\)\,.{}
\end{proof}

\begin{proof}[Proof of Proposition~\ref{propositionMinMax}]
	Take sequences \((u_{n})_{n \in \N}\) and \((v_{n})_{n \in \N}\) in \(H_{i}(\Omega)\) such that
	\[{}
	u_{n} \to u,
	\quad 
	v_{n} \to v
	\quad{}
	\text{in \(\mathcal{H}_{i}(\Omega)\) and almost everywhere in \(D_{i}\)\,.}
	\]
	For every \(n \in \N\), we have
	\[{}
	(\min{\{u_{n}, v_{n}\}})^{2} + (\max{\{u_{n}, v_{n}\}})^{2} 
	 = u_{n}^{2} + v_{n}^{2}
	\]
	and
	\[{}
	\abs{\nabla\min{\{u_{n}, v_{n}\}}}^{2} + \abs{\nabla\max{\{u_{n}, v_{n}\}}}^{2} 
	= \abs{\nabla u_{n}}^{2} + \abs{\nabla v_{n}}^{2}
	\]
	almost everywhere in \(D_{i}\).
	Thus, 
	\[{}
	\norm{\min{\{u_{n}, v_{n}\}}}_{i}^{2} + \norm{\max{\{u_{n}, v_{n}\}}}_{i}^{2}
	= \norm{u_{n}}_{i}^{2} + \norm{v_{n}}_{i}^{2}.
	\]
	In particular, the sequences \((\min{\{u_{n}, v_{n}\}})_{n \in \N}\) and \((\max{\{u_{n}, v_{n}\}})_{n \in \N}\) are bounded in \(\mathcal{H}_{i}(\Omega)\).{}
	Hence, we may extract subsequences that converge weakly in \(\mathcal{H}_{i}(\Omega)\) to \(f\) and \(g\), respectively.
	As they converge almost everywhere to \(\min{\{u, v\}}\) and \(\max{\{u, v\}}\) in \(D_{i}\)\,, the conclusion follows from Lemma~\ref{lemmaClosureH} and the lower semicontinuity of the norm under weak convergence.
\end{proof}

In contrast with the existence of a positive and negative parts in \(\mathcal{H}_{i}(\Omega)\), it is unclear whether every \(u \in \mathcal{H}_{i}(\Omega)\) admits a truncation \(T_{k}(u) \in \mathcal{H}_{i}(\Omega)\) for every \(k > 0\).
Let us now show that the minimizer $\theta_{i,h}$ is an upper bound for the approximating sequence constructed in Section~\ref{Approximation}.

\begin{proposition}
	\label{propositionUpperBoundMinimizer}
	If \(h \in L^{2}(D_{i}; w_{i} \dif x)\) is nonnegative and if \((u_{n})_{n \in \N}\) is the sequence defined in Proposition~\ref{increasingsequence} with $\mu = w_{i}h \chi_{D_{i}}$, then, for every \(n \in \N\), we have
\begin{equation*}
0 \le u_{n} \le \theta_{i, h} \quad \text{almost everywhere in \(D_{i}\)\,.}
\end{equation*}
\end{proposition}

\begin{proof}
Since \(u_{0} = 0\), the case \(n = 0\) is a consequence of Proposition~\ref{propPositive}.
We now assume by induction that the inequality holds for \(n - 1\) for some \(n \ge 1\).
Our goal is to show that if \((\xi_{j})_{j \in \N}\) is a minimizing sequence in \(H_{i}(\Omega)\) of the functional \(E\), then \((\max{\{u_{n}, \xi_{j}\}})_{j \in \N}\) is also a minimizing sequence.
To this end, we begin by observing that $u_{n} \in W^{1,2}_0(\Omega) \cap L^2(\Omega; V^+ \dif x)$ is the minimizer of the functional \(F_{n}\) defined on \(W_{0}^{1, 2}(\Omega) \cap L^{2}(\Omega; V^{+} \dif x)\) by
	\[
	F_{n} (\eta) = \frac{1}{2} \int_{\Omega}(\abs{\nabla \eta}^{2} + V^{+}\eta^{2}) - \int_{\Omega} \bigl(T_{n}(w_{i}h\chi_{D_{i}}) + T_{n}(V^{-}) u_{n - 1}\bigr) \eta.
	\]
In particular, for every \(\xi \in H_{i}(\Omega)\),
	\begin{equation}
		\label{eqClaimComparisonApproximation}
	F_{n}(u_{n}) 
	\le F_{n}(\min{\{u_{n}, \xi\}}). 
	\end{equation}

We next show the following consequence of \eqref{eqClaimComparisonApproximation} for the functional \(E\)\,:{}
	\begin{equation}
		\label{eq922}
		E(u_{n}) \le 
		E(\min{\{u_{n}, \xi\}}) + \int_{\Omega} V^{-} (u_{n - 1} - \xi)^{+} (u_{n} - \xi)^{+}.
	\end{equation}
Indeed, for every \(\eta \in H_{i}(\Omega)\),{}
\[{}
E(\eta){}
= F_{n}(\eta){}
- \frac{1}{2} \int_{\Omega} V^{-}\eta^{2} + \int_{\Omega} T_{n}(V^{-}) u_{n-1} \eta - \int_{D_{i}} (w_{i} h - T_{n}(w_{i} h)) \eta.
\]
Applying \eqref{eqClaimComparisonApproximation} and \(u_{n} \ge \min{\{u_{n}, \xi\}}\), we get
\[{}
\begin{split}
E(u_{n}) - E(\min{\{u_{n}, \xi\}})
& \le - \frac{1}{2} \int_{\Omega} V^{-} (u_{n}^{2} - \min{\{u_{n}, \xi\}}^{2}) + \int_{\Omega} V^{-} u_{n-1} (u_{n} - \min{\{u_{n}, \xi\}})\\
& = \int_{\Omega} V^{-} \Bigl(u_{n - 1} - \frac{u_{n} + \min{\{u_{n}, \xi\}}}{2}\Bigr) (u_{n} - \min{\{u_{n}, \xi\}}).
\end{split}
\]
Thus,
\[{}
\begin{split}
E(u_{n}) - E(\min{\{u_{n}, \xi\}})
& \le \int_{\{u_{n} > \xi\}} V^{-} \Bigl(u_{n - 1} - \frac{u_{n} + \xi}{2}\Bigr)(u_{n} - \xi) \\
& \le \int_{\Omega} V^{-} (u_{n - 1} - \xi)^{+} (u_{n} - \xi)^{+},
\end{split}
\]
which gives \eqref{eq922}.
	
	Using a standard property of the maximum and minimum between two functions, we then get from \eqref{eq922} that
	\begin{equation}
	\label{eq793}
	\begin{split}
		E(\max{\{u_{n}, \xi\}}) 
		& = E(\xi) + E(u_{n}) - E(\min{\{u_{n}, \xi\}})\\
		& \le	E(\xi) + \int_{\Omega} V^{-} (u_{n - 1} - \xi)^{+} (u_{n} - \xi)^{+}.
	\end{split}
	\end{equation}
	As we mentioned above, we now apply this estimate to a minimizing sequence \((\xi_{j})_{j \in \N}\) in \(H_{i}(\Omega)\) of the functional \(E\).
	We first recall that, by Proposition~\ref{existence-sequence}, we have \(\xi_{j} \to \theta_{i, h}\) in \(\mathcal{H}_{i}(\Omega)\).{}
	Since the functions \(u_{n}\) and \(u_{n - 1}\) are bounded and belong to \(L^{1}(\Omega; V^{-} \dif x)\), they also belong to \(L^{2}(\Omega; V^{-} \dif x)\).{}
	By the Dominated Convergence Theorem and the induction assumption \(u_{n - 1} \le \theta_{i, h}\) in \(D_{i}\)\,, we then have
	\[{}
	\lim_{j \to \infty}{\int_{\Omega} V^{-} (u_{n - 1} - \xi_{j})^{+} (u_{n} - \xi_{j})^{+}}
	= \int_{\Omega} V^{-} (u_{n - 1} - \theta_{i, h})^{+} (u_{n} - \theta_{i, h})^{+}
	= 0.
	\]
	We then deduce from \eqref{eq793} that
	\[{}
	\limsup_{j \to \infty}{E(\max{\{u_{n}, \xi_{j}\}})}
	\le \lim_{j \to \infty}{E(\xi_{j})} + 0 = \inf{E}.
	\]
	Hence, \((\max{\{u_{n}, \xi_{j}\}})_{j \in \N}\) is also a minimizing sequence of \(E\) as claimed.{}
	Therefore, 
	\(\max{\{u_{n}, \xi_{j}\}} \to \theta_{i, h}\) in \(\mathcal{H}_{i}(\Omega)\) as \(j \to \infty\)
	and then, by \eqref{eqPoincareStrong}, 
	\[{}
	u_{n} \le \max{\{u_{n}, \xi_{j}\}} \to \theta_{i, h}
	\quad \text{in \(L^{2}(D_{i}; w_{i} \dif x)\).}
	\]
	Thus, \(u_{n} \le \theta_{i, h}\) almost everywhere in \(D_{i}\)\,.
\end{proof}

We now identify the limit of the sequence \((u_{n})_{n \in \N}\) with the minimizer \(\theta_{i, h}\)\,:

\begin{proposition}
	\label{propositionLimitMinimizer}
	Let \(h \in L^{2}(D_{i}; w_{i} \dif x)\) be a nonnegative function and let \((u_{n})_{n \in \N}\) be the sequence defined in Proposition~\ref{increasingsequence} with \(\mu = w_{i}h \chi_{D_{i}}\).{}
	If \((u_{n})_{n \in \N}\) is bounded in \(L^{2}(D_{i}; w_{i} \dif x)\) and in \(L^{1}(D_{i}; V^{-} \dif x)\), then
	\[{}
	\lim_{n \to \infty}{u_{n}}
	= \theta_{i, h}
	\quad \text{almost everywhere in \(D_{i}\)\,.}
	\]
\end{proposition}

\begin{proof}
We first show that \((u_{n})_{n \in \N}\) is bounded in \(\mathcal{H}_{i}(\Omega)\).{}
To this end, we recall that, for every \(n \in \N\), we have $u_n\in H_i(\Omega)$ and
\begin{equation}\label{eqEulerLagrangeApproximation-Bis}
\int_{D_i} (\nabla u_n\cdot \nabla \xi + V^+ u_n \xi){}
= \int_{D_i} (T_{n}(w_{i} h) + T_{n}(V^{-}) u_{n-1})\xi 
\end{equation}
for every \(\xi \in W_{0}^{1, 2}(\Omega) \cap L^{2}(\Omega; V^{+} \dif x).\)
On the other hand, since \((u_{n})_{n \in \N}\) is nondecreasing,		
		\[
		\norm{u_{n}}_{i}^{2}
		= \int_{D_{i}} (\abs{\nabla u_{n}}^{2} + V u_{n}^{2})
		\le \int_{D_{i}} (\abs{\nabla u_{n}}^{2} + V^{+} u_{n}^{2} - T_{n}(V^{-}) u_{n - 1} u_{n}).
		\]
		Taking \(\xi = u_{n}\) in \eqref{eqEulerLagrangeApproximation-Bis}, we get
		\[
		\norm{u_{n}}_{i}^{2}
		\le \int_{D_{i}} T_{n}(w_{i}h) u_{n}
		\le \int_{D_{i}} w_{i} h u_{n}.
		\]
		Since \(h \in L^{2}(D_{i}; w_{i} \dif x)\) and \((u_{n})_{n \in \N}\) is bounded in \(L^{2}(D_{i}; w_{i} \dif x)\), it follows from this estimate that \((u_{n})_{n \in \N}\) is also bounded in \(\mathcal{H}_i(\Omega)\). Therefore, there exists a subsequence $(u_{n_{j}})_{j \in \N}$ such that
		\[
			u_{n_{j}} \rightharpoonup \widetilde{u} \quad \text{in \(\mathcal{H}_i(\Omega)\).} 
		\]
		Since $(u_{n})_{n \in \N}$ is nondecreasing, it converges pointwise and then, by Lemma~\ref{lemmaClosureH},
		\begin{equation*}
		\widetilde{u} 
		= \lim_{j \to \infty}{u_{n_{j}}}
		= \lim_{n \to \infty}{u_{n}}
		\quad \text{almost everywhere in \(D_{i}\)\,.}
		\end{equation*}
		As \((u_{n})_{n \in \N}\) is bounded in \(L^{1}(D_{i}; V^{-} \dif x)\), we also have \(\widetilde{u} \in L^{1}(D_{i}; V^{-} \dif x)\).
	
	Recalling that \(E\) has a unique minimizer \(\theta_{i, h}\) in \(\mathcal{H}_i(\Omega)\), to conclude the proof of the proposition it suffices to show that \(\widetilde u\) satisfies the Euler-Lagrange equation:
	\begin{equation}
	\label{eqEulerLagrangeDirichlet-Bis}
	\ps{\widetilde u}{\xi}_{i}
	= \int_{D_i} w_{i} h\xi
	\quad \text{for every \(\xi \in \mathcal{H}_i(\Omega)\),}
	\end{equation}
	where \(\ps{\cdot}{\cdot}_{i}\) denotes the inner product of \(\mathcal{H}_i(\Omega)\) associated to its norm.
	Since we do not know whether \(\widetilde u \in L^{2}(D_{i}; V \dif x)\), one should be careful when letting \(n \to \infty\) in \eqref{eqEulerLagrangeApproximation-Bis}.
	
	To overcome this difficulty, we proceed with \(\xi \in H_{i}(\Omega) \cap L^{\infty}(\Omega)\).
	We rewrite \eqref{eqEulerLagrangeApproximation-Bis} with \(n = n_{j}\) as
	\begin{equation}
		\label{eqEulerLagrangeApproximation-Variant-Bis}
	\ps{u_{n_{j}}}{\xi}_{i} + \int_{D_{i}} (V^{-}u_{n_{j}} - T_{n_{j}}(V^{-}) u_{n_{j} - 1})\xi
	= \int_{D_{i}} T_{n_{j}}(w_{i}h) \xi.
	\end{equation}
	Since \(0 \le u_{n} \le \widetilde u\) in \(D_{i}\)\,, \(\widetilde{u} \in L^{1}(D_{i}; V^{-} \dif x)\) and \(\xi \in L^{\infty}(\Omega)\), we deduce from the Dominated Convergence Theorem that
	\[
	\lim_{j \to \infty}{\int_{D_{i}} (V^{-}u_{n_{j}} - T_{n_{j}}(V^{-}) u_{n_{j} - 1})\xi}
	= 0.
	\]
	As \(j \to \infty\) in \eqref{eqEulerLagrangeApproximation-Variant-Bis}, we then get
	\[
	\ps{\widetilde u}{\xi}_{i}
	= \int_{D_{i}} w_{i} h \xi
	\quad \text{for every \(\xi \in H_{i}(\Omega) \cap L^{\infty}(\Omega)\).}
	\]
	For a general \(\xi \in H_{i}(\Omega)\), we apply this identity to \(T_{k}(\xi)\) and then let \(k \to \infty\).
	Equation \eqref{eqEulerLagrangeDirichlet-Bis} then follows from the density of \(H_{i}(\Omega)\) in \(\mathcal{H}_i(\Omega)\).
	Since the minimizer \(\theta_{i, h}\) satisfies the same equation, it follows by uniqueness that \(\widetilde u = \theta_{i, h}\).{}
	In particular, \(\widetilde u = \theta_{i, h}\) almost everywhere in \(D_{i}\)\,.	
\end{proof}

We conclude this section with an example due to Vázquez and Zuazua~\cite{VazquezZuazua} where the space \(\mathcal{H}_{i}(\Omega)\) is strictly larger than $H_{i}(\Omega)$\,:

\begin{example}
	\label{exampleSupersolution}
	For \(N \ge 3\) and \(0 < \alpha < N - 2\), let \(u_{\alpha} : B_{1} \to \R\) be defined for \(x \ne 0\) by
	\begin{equation}
	\label{eq-1760}
	u_{\alpha}(x) = \frac{1}{|x|^{\alpha}} - 1
	\end{equation}
	and take \(V_{\alpha} : B_{1} \to [- \infty, 0]\) given by
	\begin{equation}
    \label{eq-1765}
    V_{\alpha}(x) 
    = \frac{\Delta u_{\alpha}(x)}{u_{\alpha}(x)}
    = - \frac{\alpha(N-2-\alpha)}{|x|^2(1-|x|^{\alpha})}.
	\end{equation}
	Then, \(u_{\alpha} \in W_{0}^{1, 1}(B_{1}) \cap L^{1}(B_{1}; V_{\alpha} \dif x)\) and
	\begin{equation}
		\label{eq-1357}
	- \Delta u_{\alpha} + V_{\alpha} u_{\alpha} = 0
	\quad \text{in the sense of distributions in \(B_{1}\).}
	\end{equation}

    Since \(S \subset \{0\}\), the set \(D \vcentcolon= B_{1} \setminus S\) is connected.
    We thus have only one completion space, which we denote by \(\mathcal{H}(D)\).{}
    From \cite{OP}*{Example~1.2}, the torsion function \(\zeta_{1, \alpha}\) satisfies \(\widehat{\zeta_{1, \alpha}}(0) = 0\), whence \(S = \{0\}\).
    Since $V_{\alpha}$ is negative, $H(D) = W_{0}^{1, 2}(B_{1})$. 
    We now verify that for \(\alpha = \frac{N-2}{2}\) we have
    \begin{equation}\label{by1362}
    u_{\alpha} \not\in H(D){}
    \quad \text{and} \quad{}
    u_{\alpha} \in \mathcal{H}(D).
    \end{equation}
    For the first assertion of \eqref{by1362}, it suffices to observe that
    \(\abs{\nabla u_{\alpha}} = \alpha/|x|^{\alpha + 1} \not\in L^{2}(B_{1})\) when \(\alpha \ge \frac{N-2}{2}\).{}
	To show that $u_{\alpha}\in \mathcal{H}(D)$ we proceed by truncation.
	We have
\begin{equation}
\label{eq-1375}
\int_{B_1}|\nabla T_k(u_{\alpha})|^2
= \alpha^2 \int_{B_{1} \setminus B_{\rho_{k}}}{\frac{\dif x}{|x|^{2\alpha + 2}}},
\end{equation}
where \(\rho_k=(k+1)^{-{1}/{\alpha}}\).
Moreover, denoting by \(O(1)\) a quantity that is uniformly bounded with respect to \(k\), we can write
\begin{equation}
\label{eq-1387}
\int_{B_{1}} V_{\alpha}T_{k}(u_{\alpha})^{2}
= - \alpha (N-2-\alpha) \int_{B_{1} \setminus B_{\rho_{k}}} \frac{\dif x}{|x|^{2 \alpha + 2}} + O(1).
\end{equation}
When \(\alpha = \frac{N-2}{2}\), the coefficients of the integrals in \eqref{eq-1375} and \eqref{eq-1387} are opposite to each other and we deduce that
\[{}
\norm{T_{k}(u_{\alpha})}^{2}
= \int_{B_1} \bigl( |\nabla T_k(u_{\alpha})|^2 + V_{\alpha}T_{k}(u_{\alpha})^{2} \bigr)
= O(1).
\]
Hence, the sequence $(T_k(u_{\alpha}))_{k\in\mathbb{N}}$ is bounded in $\mathcal{H}(D)$.
Since it converges pointwise to $u_{\alpha}$\,, using Lemma~\ref{lemmaClosureH} we conclude that $u_{\alpha}\in \mathcal{H}(D)$.
\end{example}

\section{Proof of Theorem~\ref{theoremPoincare-bisA}}
\label{sectiontheoremPoincare-bisA}

\begin{proof}[{Proof of Theorem~\ref{theoremPoincare-bisA}}] 
Let \(z_{i}\) be the duality solution of
\begin{equation}
\label{eq1800}
	\left\{
	\begin{alignedat}{2}
	-\Delta z_{i} + V^{+} z_{i} & = \mu\lfloor_{D_{i}} && \quad \text{in } \Omega,\\
	 z_{i} & = 0 && \quad  \text{on }  \partial \Omega.
	\end{alignedat}
	\right.
\end{equation}
By \eqref{eq1399}, we have
\begin{equation}\label{16/07}
\zeta_{Q(z_{i})}
\le z_{i}
\quad \text{almost everywhere in \(\Omega\).}
\end{equation}
We claim that, for every \(n \in \N\),
\begin{equation}
	\label{eq977}
u_{n} \le u \chi_{D_{i}}
\quad \text{almost everywhere in \(\Omega\),}
\end{equation}
where \(u\) is the distributional solution of \eqref{eqDirichletProblemIntroduction} and $(u_{n})_{n \in \N}$ is the sequence of duality solutions defined in Proposition~\ref{increasingsequence} with datum \(Q(z_{i})\),{} that is for \(n \ge 1\),
\[{}
	\left\{
	\begin{alignedat}{2}
	-\Delta u_{n} + V^{+} u_{n} & = T_{n}(Q(z_{i})) + T_{n}(V^{-})u_{n - 1} && \quad \text{in } \Omega,\\
	u_{n} & = 0 && \quad  \text{on }  \partial \Omega.
	\end{alignedat}
	\right.
\]
Since for every nonnegative \(f \in L^{\infty}(\Omega)\), \(\zeta_{f}\) is also nonnegative, we have
\begin{equation}
\label{eq1832}
\begin{split}
\int_{\Omega} u_{n} f
& = \int_{\Omega} \zeta_{f} (T_{n}(Q(z_{i})) + T_{n}(V^{-}) u_{n - 1})\\
& \le \int_{\Omega} \zeta_{f} Q(z_{i}) + \int_{\Omega} \zeta_{f} V^{-} u_{n - 1}.
\end{split}
\end{equation}

We prove \eqref{eq977} by induction.
Firstly, we have $u_0 = 0 \le u \chi_{D_{i}}$ by assumption on \(u\).
Assume now that, for some $n \ge 1$, the inequality holds for \(n - 1\).{}
Using the duality formulation of \(z_{i}\) and the induction assumption in \eqref{eq1832},
for every nonnegative \(f \in L^{\infty}(\Omega)\) we have
\[{}
\int_{\Omega} u_{n} f
 \le \int_{\Omega} \zeta_{Q(z_{i})} f + \int_{\Omega} \zeta_{f} V^{-} u \chi_{D_{i}}\,.
\]
Using \eqref{16/07} and the fact that \(z_{i}\) and \(u\chi_{D_{i}}\) are duality solutions of \eqref{eq1800} and \eqref{eqDirichletProblemIntroduction} with datum \(\mu\lfloor_{D_{i}}\)\,, respectively, we then get
\[
\int_{\Omega} u_{n} f
 \le \int_{\Omega} z_{i} f + \int_{\Omega} \zeta_{f} V^{-} u\chi_{D_{i}}
 = \int_{\Omega} \widehat{\zeta_{f}} \dif\mu\lfloor_{D_{i}} + \int_{\Omega} \zeta_{f} V^{-} u\chi_{D_{i}}
 = \int_{\Omega} u\chi_{D_{i}} f.
\]
Since \(f\) is an arbitrary nonnegative function in \(L^{\infty}(\Omega)\), estimate \eqref{eq977} then follows.

As a consequence of \eqref{eq977}, we have
\begin{equation}
\label{eq1883}
u_{n} = 0
\quad \text{almost everywhere in \(\Omega \setminus D_{i}\)\,.}
\end{equation}
Let us now show that, for every \(n \ge 1\),
\begin{equation}
\label{eq1889}
u_{n} > 0
\quad \text{almost everywhere in \(D_{i}\)\,.}
\end{equation}
Indeed, we observe that, by the representation formula \eqref{eqRepresentationFormula}, for almost every \(x \in \Omega\) we have
\[{}
z_{i}(x) = \int_{D_{i}} \widehat{G_{x}} \dif\mu.
\]
Since \(\mu(D_{i}) > 0\) and \(\widehat{G_{x}} > 0\) in \(D_{i}\) for \(x \in D_{i}\)\,, we then have
\[{}
z_{i} > 0
\quad \text{almost everywhere in \(D_{i}\)\,.}
\]
Thus, \(Q(z_{i}) > 0\) almost everywhere in \(D_{i}\).{}
Applying again \eqref{eqRepresentationFormula}, for almost every \(x \in \Omega\) we have
\[{}
u_{n}(x) 
= \int_{\Omega} {G_{x}} \, (T_{n}(Q(z_{i})) + T_{n}(V^{-}) u_{n - 1})
\ge \int_{D_{i}} {G_{x}} \, T_{n}(Q(z_{i})).
\]
For \(x \in D_{i}\)\,, the integrand is almost everywhere positive in \(D_{i}\) and \eqref{eq1889} follows.

Let \(n \ge 1\).{}
As \(u_{n} \in W^{1, 2}_0(\Omega)\cap L^2(\Omega; V^+ \dif x)\) is also a variational solution for the same datum, we have
\[
\int_{\Omega} (\nabla u_{n} \cdot \nabla \xi + V^{+} u_{n} \xi) 
= \int_{\Omega}(T_{n}(Q(z_{i})) + T_{n}(V^-) u_{n - 1}) \xi{}
\]
for every \(\xi \in W^{1, 2}_0(\Omega)\cap L^2(\Omega; V^+ \dif x)\).{}
By Lemma~\ref{lemmaPoincareVariational} applied to the potential \(V^{+}\) with \(f = T_{n}(Q(z_{i})) + T_{n}(V^-) u_{n - 1}\)\,, it follows from \eqref{eq1883} and \eqref{eq1889} that
\begin{equation}
	\label{eq646}
	\int_{D_{i}} (\abs{\nabla \xi}^{2} + V^{+} \xi^{2})
	\ge \int_{D_{i}} \frac{T_{n}(Q(z_{i}))+ T_{n}(V^-) u_{n - 1}}{u_{n}} \xi^{2},
\end{equation}
for every \(\xi \in W^{1, 2}_0(\Omega)\cap L^2(\Omega; V^+ \dif x)\).
Denote by \(\widetilde{u}\) the pointwise limit of the nondecreasing sequence \((u_{n})_{n \in \N}\)\,.{}
Since 
\[{}
	0 \le \widetilde u \le u \chi_{D_{i}} \in L^{1}(\Omega),
\] we have \(\widetilde u < \infty\) almost everywhere in \(D_{i}\) and then \({u_{n-1}}/{u_{n}} \to 1\)  almost everywhere in \(D_{i}\)\,.
By Fatou's lemma, letting $n \to \infty$ in \eqref{eq646} we obtain that
\[{}
	\int_{D_{i}} (\abs{\nabla \xi}^{2} + V^{+} \xi^{2})
	\ge \int_{D_{i}} \Bigr(\frac{Q(z_{i})}{\widetilde u} + V^{-} \Bigr)\xi^{2}.
\]
Hence, any measurable function \(w_{i} : D_{i} \to (0, +\infty)\) such that
\begin{equation}
	\label{eqWeightChoice}
	0 < w_{i} \le \frac{Q(z_{i})}{\widetilde u}
	\quad \text{almost everywhere in \(D_{i}\)}
\end{equation}
satisfies estimate \eqref{eqPoincareStrong}.
\end{proof}

\begin{example}    
	\label{exampleNonZero}
	Let \(u_{\alpha}\) and \(V_{\alpha}\) be given by \eqref{eq-1760} and \eqref{eq-1765}, respectively.
	We show that the assumption of Theorem~\ref{theoremPoincare-bisA} is satisfied when \(\frac{N-2}{2} < \alpha < N-2\), even though \eqref{eq-1357} holds.
	Indeed, for \(\frac{N-2}{2} \le \beta < \alpha\), we have
	\begin{equation*}
	- \Delta u_{\beta} + V_{\alpha} u_{\beta} = f_{\alpha, \beta}
	\quad \text{in the sense of distributions in \(B_{1}\),}
	\end{equation*}
	where \(f_{\alpha, \beta} \in L^{1}(B_{1})\) is a nonnegative function and \(f_{\alpha, \beta}(x) > 0\) for \(x \ne 0\).
	Indeed, by an explicit computation,
	\[
	f_{\alpha, \beta}(x) 
	= \frac{\beta(N - 2 - \beta) (1 - \abs{x}^{\alpha}) - \alpha (N - 2 - \alpha) (1 - \abs{x}^{\beta})}{\abs{x}^{\beta + 2} (1 - \abs{x}^{\alpha})}.
	\]
	As the function \(t \mapsto t(N - 2 - t)\) is decreasing in the interval \([\frac{N-2}{2}, N-2]\), for \(\beta < \alpha\) in this range we then have
	\[
	f_{\alpha, \beta}
	\ge \frac{\alpha(N - 2 - \alpha)}{\abs{x}^{\beta + 2} (1 - \abs{x}^{\alpha})} (\abs{x}^{\beta} - \abs{x}^{\alpha}) > 0
	\quad \text{for every \(x \in B_{1} \setminus \{0\}\).}
	\]
\end{example}

Observe that in Example~\ref{exampleNonZero} one has \(u_{\beta} \not\in L^{2}(B_{1}; V^{-} \dif x)\) if and only if \(\beta \ge (N-2)/2\).
We conclude this section showing that the expected spectral property can be recovered on each component $D_i$ provided the solution has enough integrability with respect to \(V^{-}\).

\begin{corollary}
	\label{propositionSupersolutionZero}
	Let \(i \in I\) be such that \eqref{eqPoincareStrong} holds for some  measurable function \(w_{i} : D_{i} \to (0, +\infty)\).{}
	If \(u\) is a duality solution of \eqref{eqDualitySolution} with \(\mu \in \mathcal{M}(\Omega)\) such that \(\mu\lfloor_{D_{i}}{} = 0\) and \(u \in L^{2}(D_{i}; V^{-}\dif x)\), then \(u = 0\) almost everywhere in \(D_{i}\)\,.
\end{corollary}

\begin{proof}
	Since \(\mu\lfloor_{D_{i}}{} = 0\), by Corollary~\ref{cutoff} and Proposition~\ref{propositionDictionary} the function \(\bar u_{i} \vcentcolon= u \chi_{D_{i}}\) satisfies
	\[
	- \Delta \bar u_{i} + V \bar u_{i} = -\tau
	\quad \text{in the sense of distributions in \(\Omega\),}
	\]
	for some diffuse measure $\tau$ such that \(|\tau|(\Omega \setminus S) = 0\). 
	By \cite{DMOP}, for every \(k > 0\) we then have
	\begin{equation*}
	\int_{\Omega} \bigl(|\nabla T_{k}(\bar u_{i})|^{2} + V \bar u_{i} T_{k}(\bar u_{i})\bigr) = -\int_{\Omega} \widehat{T_{k}(\bar{u}_{i})} \dif \tau.
	\end{equation*}
	Since $|T_{k}(\bar{u}_{i})| \leq |u|$ and $\widehat{u} = 0$ quasi-everywhere on $S$ by Proposition~\ref{propositionSZero}, we have
	\[
		\widehat{T_{k}(\bar{u}_{i})} = 0 \quad \text{quasi-everywhere on $S$.}
	\]
	Therefore, from $|\tau|(\Omega \setminus S) = 0$ it follows that
	\[
	\int_{D_{i}} \bigl(|\nabla T_{k}(u)|^{2} + V u T_{k}(u)\bigr) = 0.
	\]	
	Using \eqref{eqPoincareStrong} with \(T_{k}(u)\) and this identity, we get
	\[
	\begin{split}
	\int_{D_{i}} w_{i} T_{k}(u)^{2}
	\le \int_{D_{i}} \bigl(|\nabla T_{k}(u)|^{2} + V T_{k}(u)^{2}\bigr)
	& = \int_{D_{i}} V (T_{k}(u) - u) T_{k}(u)\\
	& \le \int_{D_{i}} V^{-} (u - T_{k}(u)) T_{k}(u).
	\end{split}
	\]
	Since \(u \in L^{2}(D_{i}; V^{-} \dif x)\), by the Dominated Convergence Theorem the right-hand side converges to \(0\) as \(k \to \infty\).
	Therefore, applying Fatou's lemma we get
	\[
	\int_{D_{i}} w_{i} u^{2} = 0.
	\]
	The conclusion follows since \(w_{i} > 0\) almost everywhere in \(D_{i}\)\,.
\end{proof}


\section{Proofs of Theorems~\ref{theoremPoincare} and \ref{theoremPoincare-bisB}}
\label{sectionlast}

\begin{proof}[Proof of Theorem~\ref{theoremPoincare}]
Given \(0 < \alpha < 1\), we apply Theorem~\ref{theoremPoincare-bisA} with potential \(V_{\alpha} \vcentcolon= V^{+} - \alpha V^{-}\) and measure \(\nu_{\alpha} \vcentcolon= \mu + (1 - \alpha) V^{-}u\).{}
We observe that 
\begin{equation}
	\label{eq1038}
	\nu_{\alpha}(D_{i}) > 0.{}
\end{equation}
This is clear when \(\mu(D_{i}) > 0\).{}
Otherwise, we have \(\mu(D_{i}) = 0\) and
\[{}
-\Delta u + V^{+}u = V^{-}u
\quad \text{as a measure in \(D_{i}\).}
\]
Applying the representation formula \eqref{eqRepresentationFormula} and the fact that \(\widehat{G_{x}} = 0\) in \(\Omega \setminus D_{i}\) for \(x \in D_{i}\)\,, we get
\[{}
u(x) = 
\int_{D_{i}} {G_{x}} V^{-} u  
\quad \text{for almost every \(x \in D_{i}\)\,.}
\]
Since \(u\) is nontrivial in \(D_{i}\)\,, the integral in the right-hand side is positive for some \(x \in D_{i}\)\,.{}
We then must have \(\int_{D_{i}} V^{-}u \dif x > 0\) and \eqref{eq1038} follows.{}

We now let \(\xi \in W_{0}^{1, 2}(\Omega) \cap L^{2}(\Omega; V^{+} \dif x)\).{}
Since \(\nu_{\alpha}(D_{i}) > 0\), from Theorem~\ref{theoremPoincare-bisA} there exists a measurable function \(w_{\alpha} : D_{i} \to (0, +\infty)\) such that
\[
\int_{D_{i}} (|\nabla \xi|^2 + V_{\alpha} \xi^{2})
\ge \int_{D_{i}} w_{i,\alpha}\xi^2 \geq 0.
\]
Taking the limit as $\alpha \to 1$, by the Monotone Convergence Theorem it follows that
\[
\int_{D_{i}}(|\nabla \xi|^2 + V\xi^2)
\geq
0.
\qedhere
\]
\end{proof}

\begin{proof}[{Proof of Theorem~\ref{theoremPoincare-bisB}}]
Using the notation of the proof of Theorem~\ref{theoremPoincare-bisA}, we take a bounded measurable function \(w_{i} : D_{i} \to (0, +\infty)\) such that \(u \in L^{2}(D_{i}; w_{i} \dif x)\) and
\begin{equation}
	\label{eq1331}
	0 < w_{i} \le \frac{Q(z_{i})}{u}
	\quad \text{almost everywhere in \(D_{i}\)\,.}
\end{equation}
Since \(0 < \widetilde u \le u\) almost everywhere in \(D_{i}\)\,, this weight satisfies \eqref{eqWeightChoice} and then \eqref{eqPoincareStrong} holds.

Denote by \((v_{n})_{n \in \N}\) the sequence defined in Proposition~\ref{increasingsequence} with datum \(w_{i} h \chi_{D_{i}}\)\,, where \(h \in L^{2}(D_{i}; w_{i} \dif x)\) and \(0 \le h \le u\).{}
We claim that, for every \(n \in \N\),
		\begin{equation}
			\label{eq1009}
			v_{n} \le u \chi_{D_{i}}
			\quad \text{almost everywhere in \(\Omega\).}
		\end{equation}
	Indeed, by \eqref{eq1331} and the fact that \(h \le u\) in \(\Omega\),
		\[{}
		0 \le w_{i} h
		\le \frac{Q(z_{i})}{u} h
		\le Q(z_{i})
		\quad \text{almost everywhere in \(D_{i}\)\,.}
		\]
		By comparison of solutions, we then get 
		\[{}
		v_{n} \le u_{n} \le u \chi_{D_{i}}
		\quad \text{almost everywhere in \(\Omega\),}
		\]
		where \((u_{n})_{n \in \N}\) is the sequence used in the proof of Theorem~\ref{theoremPoincare-bisA} and the second inequality is given by \eqref{eq977}.

As a consequence of \eqref{eq1009}, the sequence \((v_{n})_{n \in \N}\) is bounded in \(L^{2}(D_{i}; w_{i} \dif x)\) and \(L^{1}(D_{i}; V^{-} \dif x)\).{}
Then, by Proposition~\ref{propositionLimitMinimizer}, we have \(\lim\limits_{n \to \infty}{v_{n}} = \theta_{i, h}\) almost everywhere in \(D_{i}\)\,.
	To conclude, since \(v_{n}\) is a duality solution we have
	\[{}
	\int_{D_{i}} v_n f
	=\int_{D_i} (T_{n}(w_{i} h) + T_n(V^{-}) v_{n-1})\zeta_{f}
	\quad \text{for every \(f \in L^{\infty}(\Omega)\)}.
	\]
	Since \(\zeta_{f}\) is bounded, \(0 \le w_{i}h \le Cu \in L^{1}(D_{i})\) and \(0 \le T_{n}(V^{-}) v_{n - 1} \le V^{-} u \in L^{1}(D_{i})\), it then follows from the Dominated Convergence Theorem that
 	\[{}
	\int_{D_{i}} \theta_{i, h} f
	=\int_{D_i} (w_{i} h + V^{-} \theta_{i, h})\zeta_{f}.
	\]
	Hence, \(\theta_{i, h}\chi_{D_{i}}\) is a duality solution of \eqref{eqDualitySolution} with datum \(w_{i}h \chi_{D_{i}}\).{}
	As we have \(\theta_{i, h} = \theta_{i, h}\chi_{D_{i}}\) in \(\Omega\) and \(w_{i}h + V^{-} \theta_{i, h} \in L^{1}(D_{i})\), the conclusion follows from Proposition~\ref{propositionDictionary}.
	\end{proof}

The next example shows that the validity of a Poincaré inequality \eqref{eqPoincareStrong} is not enough to ensure that minimizers \(\theta_{i, h}\) are distributional solutions of an equation of the form \eqref{eq-379} for any \(\mu \in \mathcal{M}(\Omega)\).

\begin{example}
\label{exampleFailureDistributions}
Given a smooth convex open subset \(\omega \Subset \Omega\), let \(V : \Omega \to [-\infty, \infty]\) be defined on \(\Omega \setminus \partial\omega\) by
\[
V =\frac{1}{4d_{\partial\omega}^{2}}\big(\chi_{\Omega\setminus \omega} -\chi_{\omega}\big) ,
\]
where \(d_{\partial\omega}\) denotes the distance to \(\partial\omega\).{}
Since \(S\) depends only on \(V^{+}\), one shows that \(S = \partial\omega\), whence \(\Omega \setminus S\) has two connected components, namely \(D_{1} \vcentcolon= \omega\) and \(D_{2} \vcentcolon= \Omega \setminus \overline\omega\).
The Poincaré inequality \eqref{eqPoincareStrong} with positive weight holds for both of them.
This is clear on \(D_{2}\), while on \(D_{1}\) there exists \(\epsilon > 1/4\diam^{2}(\omega)\) such that
\begin{equation*}
\int_{\omega} (|\nabla \xi|^2 + V\xi^2) 
\ge \epsilon \int_{\omega} \xi^2{}
	\quad \text{for every \(\xi \in W_{0}^{1, 2}(\omega)\)}\,;
\end{equation*}
see \cite{BrezisMarcus:1997}*{Theorem~II}.
Hence, for any \(h \in L^{2}(D_{1})\), the functional \(E\) with \(i = 1\) has a minimizer \(\theta_{1, h}\) in \(\mathcal{H}_{1}(\Omega) \subset L^{2}(D_{1})\).

We claim that if \(h\) is nonnegative and \(\int_{D_{1}} h > 0\), then \(\theta_{1, h}\) cannot be a distributional solution of \eqref{eqDirichletProblemIntroduction} for any \(\mu \in \mathcal{M}(\Omega)\), thus extending the nonexistence of continuous supersolutions from \cite{BrezisMarcus:1997}*{Theorem~III}.
To this end, it suffices to show that \(V \theta_{1, h} \not\in L^{1}(D_{1})\).{}

	Since \(h\) is nonnegative, we have \(\theta_{1, h} \ge 0\) almost everywhere in \(\omega\).{}
	Moreover,
	\[{}
	-\Delta\theta_{1, h} - \frac{\theta_{1, h}}{4 d_{\partial\omega}^{2}} = h
	\quad \text{in the sense of distributions in \(\omega\).}
	\]
	Given a nonempty open set \(U \Subset \omega\) such that \(\int_{U} \theta_{1, h} > 0\), let \(v\) be the solution of 
	\[
	\left\{
	\begin{alignedat}{2}
		- \Delta v & = \frac{\theta_{1, h}}{4d_{\partial\omega}^{2}}\chi_{U} \quad && \text{in \(\omega\),}\\
		v & = 0 \quad && \text{on \(\partial\omega\).}
	\end{alignedat}
	\right.
	\]
	By comparison, we have \(0 \le v \le \theta_{1, h}\) almost everywhere in \(\omega\)\,; see \cite{bookponce}*{Proposition~6.1 and Lemma~17.6}.
	Since \(v\) is a nontrivial superharmonic function in \(\omega\), the Hopf lemma for \(v\) in a neighborhood of \(\partial\omega\) implies that
	\[{}
	\infty = \int_{\omega}\frac{v}{d_{\partial\omega}^{2}}
	\le \int_{\omega}\frac{\theta_{1, h}}{d_{\partial\omega}^{2}}.
	\]
	We deduce that \(V \theta_{1, h} \not\in L^{1}(D_{1})\) and then \(\theta_{1, h}\) cannot be a distributional solution in \(\Omega\).
\end{example}

\section*{Acknowledgments}

S.~Buccheri was partially supported by PNPD/CAPES-UnB-Brazil from grant 88887.363582/2019-00 and by the Austrian Science Fund (FWF) from projects F65 and P32788.
A.~C. Ponce is grateful for the invitation and hospitality of the Math Departements of Universidade de Brasília and ``Sapienza'' Universit\`a di Roma where part of this work was carried out.
He also acknowledges support of the Fonds de la Recherche scientifique (F.R.S.--FNRS) from grant J.0020.18 and the hospitality of the Academia Belgica in Rome.

\end{document}